\documentclass[12pt, reqno]{amsart}
\usepackage{amsmath,amssymb, amsfonts,amscd,psfrag,graphicx,stmaryrd}
\usepackage{epsfig}
\usepackage{amsthm}
\usepackage{fullpage}
\usepackage{verbatim}
\usepackage{listings}
\usepackage[table]{xcolor} %!color
\usepackage{mathtools}
\DeclarePairedDelimiter{\ceil}{\lceil}{\rceil}
\DeclarePairedDelimiter\floor{\lfloor}{\rfloor}
\newcommand{\lebn}

\usepackage[misc]{ifsym}
\usepackage{caption}
\usepackage{float}
\usepackage{amsmath}
\usepackage{subfig}

\theoremstyle{plain}
\newtheorem{proposition}[equation]{Proposition}

\newtheorem{theorem}[equation]{Theorem}

\newtheorem{corollary}[equation]{Corollary}
\newtheorem{lemma}[equation]{Lemma}

\theoremstyle{definition}

\newtheorem{definition}[equation]{Definition}
\newtheorem{remark}[equation]{Remark}
\newtheorem{example}[equation]{Example}

\newtheorem*{thm1}{Theorem \ref{thm:depth}}
\newtheorem*{thm2}{Theorem \ref{thm:col-bound}}

\numberwithin{equation}{section}

\newcommand{\R}{\mathbb{R}}

\newcommand{\N}{\mathbb{N}}
\newcommand{\B}{\mathcal{B}}
\newcommand{\C}{\mathcal{C}}
\newcommand{\M}{\mathcal{M}}
\newcommand{\FE}{\mathcal{F}}
\newcommand{\DE}{\mathcal{D}}

\newcommand{\A}{\mathcal{A}}

\newcommand{\U}{\mathcal{U}}
\newcommand{\V}{\mathcal{V}}

\newcommand{\width}{\operatorname{width}}

\newcommand{\Acy}{\operatorname{Acy}}
\newcommand{\CS}{\operatorname{CS}}
\newcommand{\depth}{\operatorname{depth}}

\newcommand{\D}{\Delta}

\makeatletter
\def\moverlay{\mathpalette\mov@rlay}
\def\mov@rlay#1#2{\leavevmode\vtop{%
   \baselineskip\z@skip \lineskiplimit-\maxdimen
   \ialign{\hfil$\m@th#1##$\hfil\cr#2\crcr}}}
\newcommand{\charfusion}[3][\mathord]{
    #1{\ifx#1\mathop\vphantom{#2}\fi
        \mathpalette\mov@rlay{#2\cr#3}
      }
    \ifx#1\mathop\expandafter\displaylimits\fi}
\makeatother
\newcommand{\cupdot}{\charfusion[\mathbin]{\cup}{\cdot}}

\usepackage[a4paper,left=2cm,right=2cm,top=3cm,bottom=3cm]{geometry}

\usepackage[colorlinks, urlcolor=black, linkcolor=black, citecolor=black]{hyperref}
\renewcommand{\mod}[1]{\mathrm{mod}\ #1}

\allowdisplaybreaks

\makeatletter
\@namedef{subjclassname@2020}{%
  \textup{2020} Mathematics Subject Classification}
\makeatother

\usepackage[usenames,dvipsnames]{pstricks}
\usepackage{epsfig}
\usepackage{pst-grad} % For gradients
\usepackage{pst-plot} % For axes
\usepackage[space]{grffile} % For spaces in paths
\usepackage{etoolbox} % For spaces in paths
\makeatletter % For spaces in paths
\patchcmd\Gread@eps{\@inputcheck#1 }{\@inputcheck"#1"\relax}{}{}
\makeatother

\begin{document}

\bibliographystyle{plain}

\title[Sectionable tournaments]{Sectionable Tournaments: Their topology and Coloring}

\author{Zakir Deniz}

\address{Department of Mathematics, Duzce University, Duzce, 81620, Turkey.}
\email{zakirdeniz@duzce.edu.tr}

	
	
\date{\today}
\thanks{ }

\begin{abstract}
We provide a detailed study of topological and combinatorial properties of sectionable tournaments. This class forms an inductively constructed family of tournaments grounded over simply disconnected tournaments, those tournaments whose fundamental groups of acyclic complexes are non-trivial. When $T$ is a sectionable tournament, we fully describe the cell-structure of its acyclic complex $\Acy(T)$ by using the adapted machinery of discrete Morse theory for acyclic complexes of tournaments.  In the combinatorial side, we demonstrate that the dimension of the complex $\Acy(T)$ has a role to play. We prove that
if $T$ is a $(2r+1)$-sectionable tournament and $d$ is the dimension of $\Acy(T)$, then the (acyclic) chromatic number of $T$ satisfies $\chi(T)\leq 2 \left( 2-1/(r+1) \right)^{\log(d+1)}-1$ where the logarithm has two as its base. %We further determine their feedback numbers $\fd(T)$ by proving the equality $\fd(T)=n-d-1$, where $n$ is the order of $T$.  

\end{abstract}
\maketitle
%%%%%%%%%%%%%%%%%%%%%%%%%%%%%%%%%%%%%%%%%%%%%%%%%%%%%%%%%%%%
\section{Introduction}\label{intro}

Topological methods and techniques have shed new lights on various combinatorial problems. In particular, the notion of topological connectivity of simplicial complexes associated to graphs plays a crucial role on either bounding or even exact determination of graph's invariants (\cite{kozlov}). The choice for the associated complex may vary depending on the nature of concerned problems. Unlike the undirected graphs,  problems surrounding directed graphs have not benefited much from such a rich interconnection.  As in the undirected case, there exists a natural way to associate a simplicial complex to a directed graph. Recall that a subset $S$ of a directed graph $G$ is said to be \emph{acyclic} (\emph{transitive}) if the induced directed subgraph contains no dicycle, and the family of all acyclic sets in $G$ forms a simplicial complex $\Acy(G)$, the \emph{acyclic complex} of $G$.\medskip

The main purpose of our current work is to demonstrate that the family of acyclic complexes of directed graphs is a good choice for the investigation of a related coloring problem as well as decycling directed graphs, at least on the class of tournaments. Among all, the topology of acyclic complexes of tournaments has already received much attention. In particular,
Burzio and Demaria~\cite{MB1} have completely characterized those tournaments whose fundamental groups of acyclic complexes are non-trivial.  They choose to call them as \emph{simply disconnected tournaments}, and their characterization relies on a natural decomposition of such tournaments. In detail, if $T$ is a simply disconnected tournament, then it can be written in the form  $T\cong R_m(P^1,\ldots,P^m)$ for some tournaments $P^1,\ldots,P^m$, where $R_m$ is a highly regular tournament on $m$ vertices for some odd integer $m\geq 3$.  
In such a decomposition, we call $P^1,\ldots,P^m$ as the \emph{blocks} of $T$.  Notice that when $P^i$ is a block, its vertices have the same (out/in)-neighbors outside of $P^i$, i.e., if $u\in V(P^i)$, then $uv\in E(T)$ for some $v\in V(T-P^i)$ if and only if $wv\in E(T)$ for every $w\in V(P^i)\setminus \{u\}$.  We inductively introduce a subfamily of tournaments by declaring that a tournament $T$ is called $m$-\emph{sectionable} if either it is itself a transitive tournament or else it is a simply disconnected tournament whose each block is a $t$-sectionable tournament for some $t\leq m$.\medskip 

In the case when $m=3$, we have already displayed in~\cite{zd}, the strong connection between the topology of acyclic complexes of $3$-sectionable tournaments and their acyclic chromatic numbers. Recall that a vertex coloring of a directed graph $G$ is acyclic if each color class induces an acyclic set, and the least integer $k$ for which $G$ admits an acyclic coloring with $k$ colors is the \emph{acyclic chromatic number} of $G$, denoted by $\chi(G)$.
We proved in~\cite{zd} that the acyclic complex of any $3$-sectionable tournament $T$ is homotopy equivalent to a wedge of spheres. In particular, the chromatic number of every $3$-sectionable tournament $T$ can be (tightly) bounded from above in terms of the highest dimension of a sphere occurring in such a decomposition.\medskip 

We will show here that there is still a chance to bound or determine some combinatorial invariants of $m$-sectionable tournaments when $m\geq 3$, even if the methods developed in~\cite{zd} are no longer sufficient. For topological calculations, we apply to a more stronger technique, namely discrete Morse theory of Forman~\cite{forman98} in order to describe fully the cell structure of $\Acy(T)$. As opposed to $3$-sectionable tournaments, the complex $\Acy(T)$ is not known to be homotopy equivalent to a wedge of spheres when $m>3$. Instead, we formulate the highest dimension of a critical cell (called as the \emph{depth} of $T$) occurring with respect to an acyclic matching of $\Acy(T)$ by using the inductive nature of sectionable tournaments. 

\begin{thm1}
Let $T=R_m(P^1,P^2,\ldots,P^m)$ be a sectionable tournament with $\depth(P^i)=d_i$, $m=2r+1$. Then,
\[ 
\depth(T)=\max_{{ i \in \{1,\ldots,m\}}} \{1,d_i+d_{i+1}+\ldots+d_{i+r}+r: \textnormal{either $d_{i},\ldots,d_{i+r-1}$  or  $d_{i+1},\ldots,d_{i+r}$ are non-zero} \}
\]
where the index of $d_{j}$ is taken modulo $m$ for $1\leq j \leq 2r+1$.
\end{thm1}

Unfortunately, the parameter $\depth(T)$ is not the right invariant to bound the chromatic number of sectionable tournaments when $m>3$. In general, we verify that the dimension $\dim(T)$ of the complex $\Acy(T)$ takes its role:

\begin{thm2}
If $T=R_m(P^1,P^2,\ldots,P^m)$ is an $m$-sectionable tournament with $m=2r+1$, then  
$$\chi(T)\leq 2 \left( 2-\frac{1}{r+1} \right)^{\log(\dim(T)+1)}-1,$$
where the logarithm has two as its base.
\end{thm2}

We further show that the gap between $\dim(T)$ and $\depth(T)$ could be arbitrarily large, whereas they coincide when $m=3$. In fact, we completely characterize sectionable tournaments for which the equality $\dim(T)=\depth(T)$ holds. \medskip

%Our final aim is to determine the feedback number of sectionable tournaments. Recall that the \emph{feedback number} $\fd(G)$ of a directed graph $G$ is the minimum cardinality of a set $A\subset V(G)$ such that $G-A$ is acyclic. Note that Vishnoi \cite{vish} has shown that the problem of finding a feedback vertex set of size at most $k$ is NP-complete even for tournaments. 

%\begin{thm3}
%$\fd(T) = n-\dim(T)-1$ for any  sectionable tournament $T$ with $n=\vert V(T) \vert$. 
%\end{thm3}

%%%%%%%%%%%%%%%%%%%%%%%%%%%%%%%%%%%%%%%%%%%%%%%%%%%%%%%%%%%%%%%%%%%%%%%%%%%%%%%%%%%%%%%%%%%
\section{Preliminaries}\label{section:prel}
%%%%%%%%%%%%%%%%%%%%%%%%%%%%%%%%%%%%%%%%%%%%%%%%%%%%%%%%%%%%%%%%%%%%%%%%%%%%%%%%%%%%%%%%%%%
We first recall some general notions and notations needed throughout the paper, and repeat some of the definitions mentioned in the introduction more formally.
\medskip

A \emph{tournament} $T$ is a (finite) directed graph on a set $V=V(T)$ such that either $uv$ or $vu$ is an edge (but not both) of $T$ 
for each pair of distinct vertices $u$ and $v$. If $S\subseteq V$, the tournament induced by $S$ is written $T[S]$. 
For a given subset $S\subseteq V$, the sets $N^+(S):=\{v\in V\colon uv\;\textrm{for\;some}\;u\in S\}$ and
$N^-(S):=\{w\in V\colon wu\;\textrm{for\;some}\;u\in S\}$ are called the set of
\emph{out-neighbors} and \emph{in-neighbors} of $S$ in $T$ respectively. Furthermore, $d^+(x):=|N^+(x)|$ and $d^-(x):=|N^-(x)|$ are  called \emph{out-degree} and \emph{in-degree} of the vertex $x$.
\medskip

A tournament $T$ is said to be \emph{regular} if $d^+(v)=d^-(v)$ for any $v\in V$, and an $n$-vertex regular tournament $T_n$ is called \emph{highly-regular} (see \cite{MB2})
if $n=2k+1$ for some $k\geq 1$ and there exists a cyclic ordering $v_1,v_2,\ldots,v_{2k+1}$ of the vertices such that $v_iv_j\in E(T_n)$
for any $i\in [n]$ and $j\equiv i+l$ (mod $n$) for some $l\in [k]$. A highly regular tournament of size $m$ for $m\geq 3$ is denoted by $R_m$. 
\medskip

For each $k\geq 3$, we denote the directed cycle (or \emph{dicycle}) by $C_k$, where it is the directed graph on an ordered set of vertices
$u_1,u_2,\ldots, u_k$ such that $u_iu_{i+1}\in E(C_k)$ in the cyclic fashion, and a subdigraph $C$ of a tournament $T$ is called
a dicycle of $T$ whenever $C\cong C_k$ for some $k\geq 3$. We denote by $R_1$,  the one point tournament.
\medskip

A subset $S\subseteq V$ is said to be an \emph{acyclic} (or \emph{transitive}) set in $T$ if the induced subtournament $T[S]$ contains no dicycle,
and a tournament $T$ is called \emph{transitive} whenever $V$ is itself a transitive set in $T$. A map $\mu\colon V\to [k]$ is said to be a 
$k$-\emph{coloring} of a tournament $T$, if the set $\{v\in V\colon \mu(v)=i\}$ is transitive in $T$ for each $i\in [k]$, and the \emph{chromatic number}
of $T$ is defined to be the least integer $k$ for which $T$ admits a $k$-coloring.\medskip

For any given $A,B\subseteq V$, we write $A \Rightarrow B$ provided that $ab\in E(T)$ for any $a\in A$ and $b\in B$, and
a tournament $T$ is said to be \emph{reducible} if there exist disjoint subsets $A,B\subset V$ satisfying $V=A\cup B$ 
such that either $A\Rightarrow B$ or $B\Rightarrow A$ holds in $T$; 
otherwise it is called \emph{irreducible}.\medskip

A subset $P\subseteq V$ is said to form an \emph{equivalent} set of vertices in $T$, if for any $q\in V\backslash P$, we have either $q\Rightarrow P$ or $P\Rightarrow q$. If the vertices of $T$ can be partitioned into disjoint subtournaments
$P^1, P^2,\ldots,P^m$ such that each $P^i$ has an equivalent set of vertices, 
and $Q_m$ is a  tournament on $w_1,w_2,...,w_m$ in which $w_iw_j\in E(Q_m)$ if and only if $P^i\Rightarrow P^j$, then we write $T=Q_m(P^1, P^2,\ldots,P^m)$ and call $T$ as the \emph{composition} of the components $P^1, P^2,\ldots,P^m$
with the \emph{quotient} $Q_m$. In particular, a tournament $T_n$ is called \emph{simple}, if $T_n=Q_m(P^1, P^2,\ldots,P^m)$ implies that
either $m=1$ or $m=n$. Note that every non-trivial tournament has precisely one non-trivial simple quotient tournament~\cite{DK}.
Furthermore, if $T=Q_m(P^1, P^2,\ldots,P^m)$, then there exists a digraph homomorphism $\Gamma_{T,Q}\colon T\to Q$ defined by
$\Gamma_{T,Q}(P^i)=w_i$ for any $i\in [m]$.
\medskip

Finally, we call a tournament $T$ is \emph{highly decomposable} if its non-trivial simple quotient is highly regular, that is, 
$T=R_m(P^1, P^2,\ldots,P^m)$, where $R_m$ is a highly regular tournament for some $m\geq 3$.

\subsection{Simplicial complexes}
An \emph{(abstract) simplicial complex} $\D$ on a finite set $X$ is a family
of subsets of $X$ satisfying the following properties.
\begin{itemize}
\item[(i)] $\{x\}\in \D$ for all $x\in X$.
\item[(ii)] If $F\in \D$ and $H\subset F$, then $H\in \D$.
\end{itemize} 

The elements of $\D$ are called \emph{faces}; the \emph{dimension}
of a face $F$ is $\textrm{dim}(F):=|F|-1$, and the \emph{dimension} of $\D$ 
is defined to be $\textrm{dim}(\D):=\textrm{max}\{\textrm{dim}(F)\colon F\in \D\}$. The $0$ and 
$1$-dimensional faces of $\D$ are called \emph{vertices} and \emph{edges} while
maximal faces are called \emph{facets}, and the set of facets of $\D$ is denoted by $\FE(\D)$.
A subset $C$ of $X$ is called a \emph{circuit} of $\D$ if $C$ is a minimal non-face of $\D$.

The simplicial \emph{join} of two complexes $\D_1$ and $\D_2$ on disjoint sets of vertices, denoted by $\D_1 * \D_2$, is the simplicial complex defined by 
$$\D_1\ast \D_2=\{\sigma \cup \tau \; \; \vert \: \: \sigma \in \D_1, \tau \in \D_2 \}$$

The following can be easily obtained by the definition of join operation on simplicial complexes.
\begin{corollary}\label{cor:join-simplical}
$\dim(\D_1*\D_2)=\dim(\D_1)+\dim(\D_2)+1$ for simplicial complexes $\D_1$ and $\D_2$.
\end{corollary}

A simplicial complex is said to be \emph{simply connected} if its fundamental group is trivial, otherwise, it is called \emph{simply disconnected}. We denote by $S^j$, the sphere of dimension $j$.

Let $\D$ be a simplicial complex on the vertex set $V$, and consider $x,y\in V$ such that $\{x,y\}\in \D$ 
and $u_{xy}\notin V$. The \emph{edge contraction} $xy\mapsto u_{xy}$ can be defined to be a map $f\colon V\to (V\backslash \{x,y\})\cup \{u_{xy}\}$ by
\begin{equation*}
f(v):=\begin{cases}
v, & \textrm{if}\;v\notin \{x,y\},\\
u_{xy}, & \textrm{if}\;v\in \{x,y\}.
\end{cases}
\end{equation*} 
We then extend $f$ to all simplices $F=\{v_0,v_1,\ldots,v_k\}$ of $\D$ by setting $$f(F):=\{f(v_0),f(v_1),\ldots,f(v_k)\}.$$
The simplicial complex $\D\sslash xy:=\{f(F)\colon F\in \D\}$ is called the contraction of $\D$ with respect to the edge $\{x,y\}$.

Observe that the contraction of an edge may not need to preserve the homotopy type of the complex in general. However, under a suitable restriction on the contracted edge, we can guarantee this to happen. 

\begin{theorem}\textnormal{\cite[Theorem 2.4]{EH}, \cite[Theorem 1]{ALS}}\label{thm:hmtpy-edge}
Let $\D$ be a simplicial complex and let $\{x,y\}\in \D$ be an edge. Then, the simplicial complexes $\D$ and $\D\sslash xy$ are homotopy equivalent
provided that no minimal non-face of $\D$ contains the edge $\{x,y\}$. 
\end{theorem}
%%%%%%%%%%%%%%%%%%%%%%%%%%%%%%%%%%%%%%%%%%%%%%%%%%%%%%%%%%%%%%%%%%%%%%%%%%%%%%%%%%%%%%%

\subsection{Discrete Morse Theory}
We close this section by a short incomplete introduction to discrete Morse theory of Forman~\cite{forman98}, and refer readers to~\cite{kozlov} for more details. 

For a simplicial complex $\D$, we can associate to it a poset $P(\D)$ called the face poset of $\D$, which is the set of faces of $\D$ ordered by inclusion.  We regard any  finite poset $P(\D)$ as a directed graph, by considering the Hasse diagram of $P(\D)$ with edges pointing downward (that is, from large faces to smaller ones). A set $M$ of pairwise disjoint edges of this digraph, where no face is contained in more than one pair in $M$, is called a \emph{matching} of $P(\D)$. A face is \emph{unmatched} (or \emph{critical}) if it is not matched with any element in $P(\D)$. A matching is called \emph{perfect} if it covers all elements of $P(\D)$, and a matching is said to be a \emph{Morse matching} (or an \emph{acyclic matching}) if the directed
graph obtained from $P(\D)$ by reversing the direction of the edges in $M$
is acyclic. We need the following classical results of discrete Morse theory \cite{forman02}.

\begin{theorem}\label{thm:Forman}
Let $\Sigma$ be a finite simplicial complex, seen as a poset,
and $M$ a Morse matching on $\Sigma$,
such that the element $\emptyset$ of $\Sigma$ is matched. For $i\ge 0$, let $n_i$ be
the number of unmatched $i$-dimensional elements of $\Sigma$.
Then there exists a CW-complex having $1+n_0$ $0$-dimensional cells and $n_i$
$i$-dimensional cells for $i\ge 1$ that is homotopy equivalent to $\Sigma$.
\end{theorem}

\begin{corollary}\label{cor:Forman}
Under the assumptions of Theorem~{\rm\ref{thm:Forman}}, if all
unmatched elements in $\Sigma$ have the same dimension $i>0$  and there
are $j$ of them,
 $\Sigma$ is homotopy equivalent to a wedge of $j$  spheres of dimension $i$.
In particular, if $\Sigma$ is perfectly matched, then it is contractible.
\end{corollary}

\begin{lemma}\label{lem:morse}
Let $\D_0$ and $\D_1$ be disjoint families of subsets of a finite set such that $\tau \not \subseteq \sigma$ if $\sigma \in \D_0$ and $\tau \in \D_1$. If $M_i$ is an acyclic matching on $\D_i$ for $i=0,1,$ then $M_0 \cup M_1$ is an acyclic matching on $\D_0 \cup \D_1$.
\end{lemma}

%%%%%%%%%%%%%%%%%%%%%%%%%%%%%%%%%%%%%%%%%%%%%%%%%%%%%%%%%%%%%%%%%%%%%%%
\section{Discrete Morse theory for Acyclic complexes of tournaments}

In this section, we adapt the language of discrete Morse theory for its use in describing the CW-complex structure of acyclic complexes of sectionable tournaments. Our adaptation mainly follows the spirit of results presented in \cite{bousquet,taylan}.\medskip

We begin with outlining the general principle that we use to construct the Morse
matchings for acyclic complex $\Acy(T)$ of a tournament $T$ with a vertex set $V$.  Pick a vertex $p \in V(T)$ and define $$\DE(p)=\{ I \in \Acy(T): \ \exists \ \{x,y\} \subseteq I  \ \text{such that} \ T[p,x,y] \cong {R_3} \}$$
and 
$$
\Delta_p= \{ I \in \Acy(T): \ \text{if} \ H \in \DE(p), \text{ then} \ H \not \subseteq I \}.
$$
It can be easily observed that $\Delta_p$ is equal to $\Acy(T)\setminus \DE(p)$.
The set of pairs $(I,I\cup\{p\})$, for $I\in \D_p,$ and $p\not \in I$, forms a perfect matching of $\D_p$, and hence it is a matching of $\Acy(T)$. The vertex $p$ is called \emph{pivot} of this matching. Note that $\DE(p)$ is exactly the set of  unmatched elements of $\Acy(T)$. There may be many unmatched elements, but we can now choose
another pivot $p'$  to match some elements
of $\DE(p)=\Acy(T)\setminus \D_p$, and we repeat this operation as long as we
can. Note that, the obtained matchings are always acyclic on the set $\Acy(T)$ by Lemma \ref{lem:morse}.

\begin{definition} 
Let $\A=\{\gamma_1,\gamma_2,\ldots,\gamma_k\}$ and $\B=\{\beta_1,\beta_2,\ldots,\beta_r\}$ be two families of sets. We define \emph{dot-join operation} of the families $\A$ and $\B$ as follows:
$$ \A \cupdot \B = \left\lbrace \gamma_i \cup \beta_j \ : \ \gamma_i \in \A, \, \beta_j \in \B \right\rbrace$$
In particular, $\A \cupdot \B=\emptyset$ whenever $\A=\emptyset$ or $\B=\emptyset$.  We call two families $\A$ and $\B$ to be \emph{equivalent}, denoted by $\A\equiv \B$, whenever
they contain the same number of sets of size $k$ for each $k\geq 0$.
\end{definition}

As an example, if we consider the families $\A=\{\{1,2\},\{2,3\}\}$ and $\B=\{\emptyset,\{3\},\{1,4\}\}$, then
$ \A \cupdot \B = \left\lbrace \{1,2\},\{2,3\}, \{1,2,3\}, \{1,2,4\},  \{1,2,3,4\} \right\rbrace $.\medskip

Notice that if $\A, \B$ and  $\C$ are families of sets, then  $\A\cupdot (\B\cup \C)=(\A\cupdot \B)\cup (\A \cupdot \C)$ holds.\medskip

\begin{definition} 
Let $T$ be a tournament and $\Acy(T)$ be its acyclic complex. We define, $$\Sigma(A:B)_{T}=\{ I \in \Acy(T): \ A \subseteq I \text{ and } B \cap I=\emptyset  \}$$
where $A$ and $B$ are two subsets of $V(T)$ such that $A \cap B=\emptyset$. In particular,  we will use $\Sigma(A:B)$ instead of $\Sigma(A:B)_{T}$ if it is clear in the context. Observe also that  $\Sigma(\emptyset:\emptyset)\equiv \Acy(T)$.
\end{definition}

We may restate the definition of $\Sigma(A:B)_{T}$ when $A,B$ are some collections of faces in $\Acy(T)$. Let $\alpha,\beta_1,\beta_2,\ldots,\beta_k \in \Acy(T)$ and $\alpha\cap \beta_i =\emptyset$ for $i\in [k]$. We define,
$$\Sigma(\alpha  :  \beta_1,\beta_2,\ldots,\beta_k)_{T}=\{ I \in \Acy(T): \ \alpha \subseteq I \text{ and } \beta_j \nsubseteq I \text{ for } 1\leq j\leq k \}$$

Also, if a face $\beta_i\in \Acy(T)$  consists of a unique vertex, say $\beta_i=\{v\}$, then we use $\{v\}$ without brackets in the expression  $\Sigma(\alpha  :  \beta_1,\beta_2,\ldots,\beta_k)_{T}$. 

%The definition of $\Sigma(A:B)_{T}$ can be stated as following when $A$ is a face of $\Acy(T)$ and $B$ is ..

In particular, for a tournament $T$ and a  vertex $v\in V(T)$, we write $\DE(v)=\langle\beta_1,\beta_2,\ldots,\beta_k\rangle_T$ for an ordering $\beta_1,\beta_2,\ldots,\beta_k \in \Acy(T)$ 
if $\DE(v)$ is the family of simplices generated by the faces $\beta_1,\beta_2,\ldots,\beta_k \in \Acy(T)$.
Namely, if  $\DE(v)=\Sigma(\beta_1:v)_T\cup\Sigma(\beta_2:v,\beta_1)_T\cup \cdots \cup\Sigma(\beta_k:v,\beta_1,\ldots,\beta_{k-1})_T$, we write  $\DE(v)=\langle\beta_1,\beta_2,\ldots,\beta_k\rangle_T$.
%This means that $\DE(v)$ is the set of elements of $\Acy(T)$ containing one of $\beta_1,\beta_2,\ldots,\beta_k$ but not $v$. 

Given  a family $\A$ of sets  and an acyclic matching $\M$ via the sequence of pivots $(p_1,p_2,\ldots,p_k)$, we denote by $\CS(\A,\M)$ (or $\CS(\A)$ for short) the family of unmatched (critical) elements of $\A$ with respect to $\M$.

\begin{example}\label{ex:t5}
Consider the tournament $T$ depicted in the Figure \ref{fig.ex-acy}.
We start with $p_1=8$ as the first pivot, and let $\M_1$ be the set of matched pairs with the pivot $p_1$. Then, the set of unmatched elements is $$\DE(p_1)=\langle 67,29,12,13,14,15 \rangle$$ 
in which we abbreviate $\{a,b\}$ to $ab$ for simplicity. Clearly, $\DE(p_1)$ is the union of the following sets.
\begin{align*}
\Sigma(67:8)& =\{67,167,267,367,467,567,679,3679,4679,5679,1679,2367,2467 \\
 & \qquad 2567,3567,3467,4567,35679,34679,45679,23467,23567,24567 \} \\
\Sigma(92:8,67)&= \{ 29,129 \} \\
\Sigma(12:8,67,92)&= \{12,123,124,125,1234,1235,1245 \} \\ 
\Sigma(13:8,67,92,12)&= \{ 13,134,135 \} \\
\Sigma(14:8,67,92,12,13)&= \{ 14,145 \} \\
\Sigma(15:8,67,92,12,13,14)&= \{ 15 \} 
\end{align*}
We next select another pivot $p_2=9$ for the remaining unmatched elements. The set of matched elements via the pivot $p_2$ (forming a matching, say $\M_2$) is as follows. 
$$\{12,129,67,679,167,1679,367,3679,467,4679,567,
5679,3567,35679,3467,34679,4567,45679 \}.$$
We may continue by choosing pivots $p_3=3,\ p_4=2$  consecutively for the unmatched elements, and construct at each round a matching $\M_i$ for $3\leq i\leq 4$. After those iterations, the matched elements are  
\begin{align*}
&\{267,2367, 2467, 23467, 2567,23567, 14,134,15,135, 124,1234,125,1235 \} \textnormal{ and }\\
& \{13,123,145,1245 \},
\end{align*}
respectively.\medskip

Consequently, the only unmatched elements left are $29$ and $24567$. By Lemma \ref{lem:morse}, $\M=\M_1\cup \M_2 \cup \M_3\cup \M_4$ is acyclic. Therefore, $\CS(\Acy(T),\M)\equiv \{\{2,9\},\{2,4,5,6,7\}\}$  so that $\Acy(T)$
is homotopy equivalent to a CW-complex with one cell in each dimensions $0, 1$ and $4$.
\end{example}

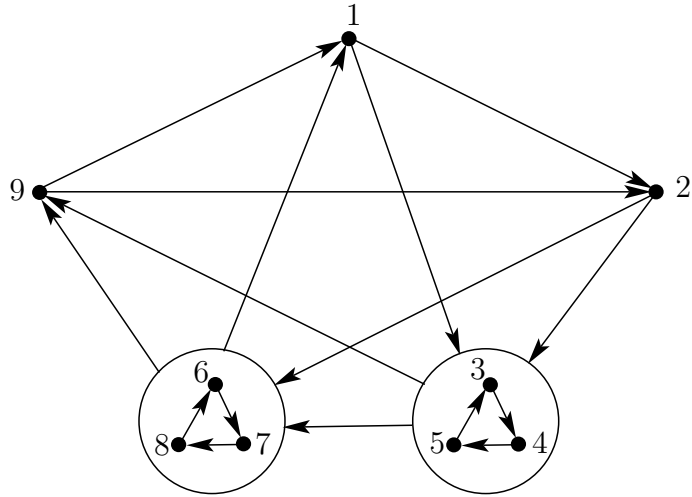
\begin{figure}[h!]
\begin{center}
\psscalebox{1.0 1.0} % Change this value to rescale the drawing.
{
\begin{pspicture}(0,-3.2161944)(8.909109,3.2161944)
\pscircle[linecolor=black, linewidth=0.02, dimen=outer](2.6611335,-2.2561944){0.96}
\pscircle[linecolor=black, linewidth=0.02, dimen=outer](6.2611337,-2.2561944){0.96}
\psdots[linecolor=black, dotsize=0.2](0.39621556,0.77724826)
\psdots[linecolor=black, dotsize=0.2](8.501134,0.78380567)
\psdots[linecolor=black, dotsize=0.2](4.462821,2.8157659)
\psline[linecolor=black, linewidth=0.02, arrowsize=0.14cm 2.0,arrowlength=2.0,arrowinset=0.2]{->}(4.49867,2.8159072)(8.459961,0.8288106)
\psline[linecolor=black, linewidth=0.02, arrowsize=0.14cm 2.0,arrowlength=2.0,arrowinset=0.2]{->}(8.466804,0.7234342)(6.8159723,-1.4974847)
\psline[linecolor=black, linewidth=0.02, arrowsize=0.14cm 2.0,arrowlength=2.0,arrowinset=0.2]{->}(4.485681,2.7605798)(5.9256496,-1.3813556)
\psline[linecolor=black, linewidth=0.02, arrowsize=0.14cm 2.0,arrowlength=2.0,arrowinset=0.2]{->}(8.4713,0.7545193)(3.474037,-1.7813556)
\psline[linecolor=black, linewidth=0.02, arrowsize=0.14cm 2.0,arrowlength=2.0,arrowinset=0.2]{->}(5.2933917,-2.3103878)(3.5901659,-2.3361943)
\psline[linecolor=black, linewidth=0.02, arrowsize=0.14cm 2.0,arrowlength=2.0,arrowinset=0.2]{->}(1.9695542,-1.6119286)(0.41438738,0.68516135)
\psline[linecolor=black, linewidth=0.02, arrowsize=0.14cm 2.0,arrowlength=2.0,arrowinset=0.2]{->}(2.8159723,-1.3297427)(4.428663,2.7450569)
\psline[linecolor=black, linewidth=0.02, arrowsize=0.14cm 2.0,arrowlength=2.0,arrowinset=0.2]{->}(0.42073658,0.83740366)(4.3805585,2.790754)
\psline[linecolor=black, linewidth=0.02, arrowsize=0.14cm 2.0,arrowlength=2.0,arrowinset=0.2]{->}(5.4611335,-1.7684524)(0.450391,0.7386403)
\psline[linecolor=black, linewidth=0.02, arrowsize=0.14cm 2.0,arrowlength=2.0,arrowinset=0.2]{->}(0.47255373,0.78564155)(8.440997,0.7848223)
\psdots[linecolor=black, dotsize=0.2](6.322672,-1.7720917)
\psdots[linecolor=black, dotsize=0.2](6.6919026,-2.5618353)
\psdots[linecolor=black, dotsize=0.2](5.840621,-2.5720918)
\psline[linecolor=black, linewidth=0.02, arrowsize=0.14cm 2.0,arrowlength=2.0,arrowinset=0.2]{->}(5.826046,-2.596923)(6.2670717,-1.8174359)
\psline[linecolor=black, linewidth=0.02, arrowsize=0.14cm 2.0,arrowlength=2.0,arrowinset=0.2]{->}(6.6638327,-2.568853)(5.9421053,-2.5791094)
\psline[linecolor=black, linewidth=0.02, arrowsize=0.14cm 2.0,arrowlength=2.0,arrowinset=0.2]{->}(6.326451,-1.7758704)(6.660054,-2.4808638)
\psline[linecolor=black, linewidth=0.02, arrowsize=0.14cm 2.0,arrowlength=2.0,arrowinset=0.2]{->}(2.708907,-1.7723616)(3.04251,-2.477355)
\psline[linecolor=black, linewidth=0.02, arrowsize=0.14cm 2.0,arrowlength=2.0,arrowinset=0.2]{->}(3.0462887,-2.565344)(2.3245614,-2.5756006)
\psline[linecolor=black, linewidth=0.02, arrowsize=0.14cm 2.0,arrowlength=2.0,arrowinset=0.2]{->}(2.208502,-2.606529)(2.6495275,-1.8270419)
\psdots[linecolor=black, dotsize=0.2](2.2230768,-2.568583)
\psdots[linecolor=black, dotsize=0.2](3.074359,-2.5583265)
\psdots[linecolor=black, dotsize=0.2](2.7051282,-1.7685829)
\rput[bl](4.420243,2.9961944){$1$}
\rput[bl](8.7591095,0.71765184){$2$}
\rput[bl](6.05823,-1.6988077){$3$}
\rput[bl](6.872386,-2.6755772){$4$}
\rput[bl](5.520662,-2.6821108){$5$}
\rput[bl](2.4190283,-1.7340164){$6$}
\rput[bl](3.2291918,-2.6897333){$7$}
\rput[bl](1.8996091,-2.7006226){$8$}
\rput[bl](0.0,0.6820243){$9$}
\end{pspicture}
}
\end{center}
\caption{The tournament $T=R_5(R_1,R_1,R_3,R_3,R_1)$ in Example~\ref{ex:t5}}
\label{fig.ex-acy}
\end{figure}

\begin{proposition} \label{prop:disjoint ground}
If $\A$ and $\B$ are two families on disjoint ground sets, then $\CS(\A \cupdot \B) \equiv \CS(\A)\cupdot \CS(\B)$.
\end{proposition}

\begin{proof}
We write $\A=\{A_1,A_2,\ldots,A_k\}$ and $\B=\{B_1,B_2,\ldots,B_r\}$ so that 
$$\A \cupdot \B = \{ A_1B_1,A_1B_2,\ldots,A_1B_r,A_2B_1,A_2B_2,\ldots,A_2B_r,\cdots,A_kB_1,A_kB_2,\ldots,A_kB_r \},$$ 
where $A_iB_j$ stands for $A_i\cup B_j$. Let $\CS(\A)$ and $\CS(\B)$ be determined by acyclic matchings $\M_1$ and $\M_2$ via the sequences of pivots $(p_1,p_2,\ldots,p_{t_1})$ and $(r_1,r_2,\ldots,r_{t_2})$ respectively. We first observe that for $A_i, A_j \in \A$, if the pair $A_i$ and $A_j$ is matched in $\A$ by $\M_1$ via the pivot $p_k$, then for each $B_r \in \B$, $A_iB_r$ and $A_jB_r$ can be matched in $\A \cupdot \B$ via the pivot $p_k$. Similarly, for $B_i, B_j \in \B$, if the pair $B_i$ and $B_j$ is matched in $\B$ by $\M_2$ via the pivot $r_s$, then for each $A_r \in \A$, $A_rB_i$ and $A_rB_j$ can be matched in $\A \cupdot \B$ via the pivot $r_s$. \medskip

We then define $\M_1'$ as the set of pairs $(A_iB_r,A_jB_r)$ for each pair $(A_i,A_j)$ in $\M_1$ with the sequence of pivots $(p_1,p_2,\ldots,p_{t_1})$ and for each $B_r\in \B$. In addition, we define $\M_2'$ as the set of pairs $(A_iB_r,A_iB_t)$ for each pair $(B_r,B_t)$ in $\M_2$ with the sequence of pivots $(r_1,r_2,\ldots,r_{t_2})$ and for each $A_i\in \CS(\A,\M_1)$. Thus $\M_1'\cup \M_2'$ is an acyclic matching on  $\A \cupdot \B$ with the sequence of pivots $(p_1,p_2,\ldots,p_{t_1},r_1,r_2,\ldots,r_{t_2})$  and the set of critical elements is $ \CS(\A)\cupdot \CS(\B)$.
%Conversely, assume that $\CS(\A \cupdot \B)$ is determined by an acyclic matching $\M$ via pivots $\{p_1,p_2,\ldots,p_{t_1}\}$. Then the unmatched elements on $\CS(\A \cupdot \B)$ are precisely  $A_i \cupdot B_j$ form for critical elements $\A_i \in   \CS(\A)$ and $\B_j \in \CS(\B)$.
\end{proof}

\begin{lemma}\label{lem:highly-reg-S1}
There exist $a,b \in V(R_{2r+1})$ such that $\CS(\Acy(R_{2r+1})) \equiv \{\{a,b\}\}$ for every highly regular tournament $R_{2r+1}$ with $r\geq 1$.
\end{lemma}

\begin{proof}
Let $V(R_{2r+1})=\{v_1,v_2,\ldots,v_{2r+1}\}$. We will construct an acyclic matching on $\Acy(R_{2r+1})$. First, select $p_1=v_1$ as a pivot, and let $\M_1$ be the set of pairs $(\gamma,\gamma\cup \{v_1\})$ for $\gamma \in \Acy(R_{2r+1})$ with $v_1 \notin \gamma$. It follows that  the set of unmatched elements with respect to $\M_1$ is  
\begin{align*}
\DE(v_1)=\langle v_2v_{r+2},v_3v_{r+2},&v_3v_{r+3},v_4v_{r+2},v_4v_{r+3},v_4v_{r+4},\ldots,\\
&v_{r+1}v_{r+2}, v_{r+1}v_{r+3},v_{r+1}v_{r+4},\cdots,v_{r+1}v_{2r+1} \rangle . 
\end{align*}

\begin{figure}[h!]
\begin{center}
\psscalebox{1.0 1.0} % Change this value to rescale the drawing.
{
\begin{pspicture}(0,-2.7661104)(5.6168423,2.7661104)
\psdots[linecolor=black, dotsize=0.2](3.0721805,2.3127453)
\psdots[linecolor=black, dotsize=0.2](4.500752,1.7202641)
\psdots[linecolor=black, dotsize=0.2](5.0721803,0.29169264)
\psdots[linecolor=black, dotsize=0.2](4.786466,-1.7083074)
\psdots[linecolor=black, dotsize=0.2](3.7578948,-2.279736)
\psdots[linecolor=black, dotsize=0.2](2.3293233,-2.279736)
\psdots[linecolor=black, dotsize=0.2](1.1864662,-1.7083074)
\psdots[linecolor=black, dotsize=0.2](0.70977443,0.20748211)
\psdots[linecolor=black, dotsize=0.2](1.3578948,1.7202641)
\psline[linecolor=black, linewidth=0.02](1.3578948,1.714249)(3.0421052,2.2931964)
\psline[linecolor=black, linewidth=0.02](0.70526314,0.25109115)(3.0736842,2.28267)
\psline[linecolor=black, linewidth=0.02](1.1894736,-1.685751)(3.0631578,2.2721438)
\psline[linecolor=black, linewidth=0.02](2.336842,-2.2962773)(3.0526316,2.28267)
\psline[linecolor=black, linewidth=0.02](3.768421,-2.254172)(3.0736842,2.28267)
\psline[linecolor=black, linewidth=0.02](4.768421,-1.6962773)(3.0947368,2.3037226)
\psline[linecolor=black, linewidth=0.02](5.0526314,0.31424904)(3.0842106,2.3353016)
\psline[linecolor=black, linewidth=0.02](4.5157895,1.7247753)(3.0526316,2.3247755)
\rput[bl](3.0,2.5661104){$v_1$}
\rput[bl](4.726316,1.6931964){$v_2$}
\rput[bl](5.336842,0.16688062){$v_3$}
\rput[bl](5.031579,-1.9068036){$v_r$}
\rput[bl](3.6359434,-2.7376122){$v_{r+1}$}
\rput[bl](1.9858793,-2.7661104){$v_{r+2}$}
\rput[bl](0.77894735,-2.1173298){$v_{r+3}$}
\rput[bl](0.0,0.14582798){$v_{2r}$}
\rput[bl](0.67368424,1.8616174){$v_{2r+1}$}
\rput{14.64782}(-0.18570182,-0.22139364){\rput[bl](0.76842105,-0.8331194){$\vdots$}}
\rput{-10.619656}(0.24643539,0.9221513){\rput[bl](5.0842104,-0.86469835){$\vdots$}}
\psline[linecolor=black, linewidth=0.02, arrowsize=0.14cm 2.0,arrowlength=2.0,arrowinset=0.2]{->}(2.1578948,1.9819524)(2.430622,2.0864978)
\psline[linecolor=black, linewidth=0.02, arrowsize=0.14cm 2.0,arrowlength=2.0,arrowinset=0.2]{->}(1.839713,1.227407)(2.0260766,1.3910433)
\psline[linecolor=black, linewidth=0.02, arrowsize=0.14cm 2.0,arrowlength=2.0,arrowinset=0.2]{->}(2.0897129,0.22740693)(2.2260766,0.51831603)
\psline[linecolor=black, linewidth=0.02, arrowsize=0.14cm 2.0,arrowlength=2.0,arrowinset=0.2]{->}(2.6578948,-0.2089567)(2.6942585,-0.027138522)
\psline[linecolor=black, linewidth=0.02, arrowsize=0.14cm 2.0,arrowlength=2.0,arrowinset=0.2]{->}(3.8078947,2.0137706)(3.9442585,1.959225)
\psline[linecolor=black, linewidth=0.02, arrowsize=0.14cm 2.0,arrowlength=2.0,arrowinset=0.2]{->}(4.098804,1.2955887)(4.189713,1.2046796)
\psline[linecolor=black, linewidth=0.02, arrowsize=0.14cm 2.0,arrowlength=2.0,arrowinset=0.2]{->}(3.930622,0.3092251)(3.9988039,0.14104329)
\psline[linecolor=black, linewidth=0.02, arrowsize=0.14cm 2.0,arrowlength=2.0,arrowinset=0.2]{->}(3.4442585,-0.15441126)(3.480622,-0.38622943)
\end{pspicture}
}
\end{center}
\caption{ }
\label{fig.sectionable2}
\end{figure}

Now, we select $v_{r+1}$ as another pivot on the family $\DE(v_1)$. By the structure of highly regular tournaments, any ordered $r+1$ elements forms a facet (see Figure \ref{fig.sectionable2}). So, for $\gamma \in \DE(v_1)$ with $v_{r+1} \notin \gamma$, if $\{v_i,v_j\} \subset \gamma$ for $2 \leq i \leq r$ and $r+2\leq j \leq 2r$, then the set of pairs $\big( \gamma, \gamma \cup \{v_{r+1}\}  \big)$ forms an acyclic matching, say $\M_2$. In this case, the unmatched elements with respect to $\M_2$ can be given as follows.
$$\U=\DE(p_1) \setminus \{\gamma \in \DE(p_1) : \{v_i,v_j\} \subset \gamma \text{ for }  2 \leq i \leq r\;\textnormal{and}\; r+2\leq j \leq 2r  \}$$
This implies that $ \beta \in \U$ if and only if $\{v_{r+1},v_k\} \subset \beta$ for  $ r+2\leq k \leq 2r+1$ and $\{v_i,v_j \} \not \subset \beta$ for each $2\leq i \leq r$ and $r+1\leq j \leq 2r$. \medskip

Let $v_{r+2}$ be the last pivot on the family $\U$. Similar to the above cases,  it is clear that $\beta \cup \{v_{r+2}\}$ is a face for every $\beta \in \U$. Therefore, if $\beta \in \U$ with $v_{r+2} \notin \beta$, we can say that the set of pairs $\big( \beta, \beta \cup \{v_{r+2}\}  \big)$  forms an acyclic matching, say $\M_3$. Notice that the only remaining unmatched element is  $\{v_{r+1},v_{r+2}\}$. We thus obtain an acyclic matching $\M=\M_1\cup \M_2 \cup \M_3$ by Lemma \ref{lem:morse}. Consequently, we obtain $\CS(\Acy(R_{2r+1}),\M) \equiv \{\{v_{r+1},v_{r+2}\}\}$ so that $\Acy(R_{2r+1})$ is homotopy equivalent to a CW-complex with one $0-\text{cell}$ and one $1-\text{cell}$. 
\end{proof}

%%%%%%%%%%%%%%%%%%%%%%%%%%%%%%%%%%%%%%%%%%%%%%%%%%%%%%%%%%%%%%%%%%%%%%%%%%%%%%%%%%%%%%%%%%%%%%%%%%
\section{Topology of Sectionable Tournaments}
%%%%%%%%%%%%%%%%%%%%%%%%%%%%%%%%%%%%%%%%%%%%%%%%%%%%%%%%%%%%%%%%%%%%%%%%%%%%%%%%%%%%%%%%%%%%%%%%%%
This section is devoted to a complete description of the cell-structure of acyclic complexes of sectionable tournaments. Our main machinery is the powerful tool of discrete Morse theory. In particular, we provide an inductive calculation of the highest dimension of a critical cell occurring in a CW-complex structure of $\Acy(T)$. 

We begin with a proper definition of sectionable tournaments. 

\begin{definition} Let $T$ be a tournament and $m\geq 3$ be an odd integer. We then call $T$ an $m$-\emph{sectionable tournament}, if either it is a transitive tournament or else it is a simply disconnected tournament $T=R_m(P^1,P^2,\ldots,P^m)$ such that each block $P^i$ is a $t$-sectionable tournament for some $t\leq m$. When there is no confusion, we sometimes drop the parameter $m$ from our notation, and use the term sectionable tournament.
\end{definition}

Notice that every set of ordered $r+1$ elements create a face in a highly regular tournament $R_{2r+1}$. Thus, we conclude the following by using Corollary \ref{cor:join-simplical}.

\begin{proposition}\label{prop:dim}
Let $T=R_{2r+1}(P^1,P^2,\ldots,P^{2r+1})$ be a tournament such that $\dim(P^i)=d_i$ for every $i\in [2r+1]$. Then,
\[ 
\dim(T)=\max_{{ i \in \{1,2,\ldots,2r+1\}}} \{d_i+d_{i+1}+\ldots+d_{i+r}+r \}
\]
where the index of $d_{j}$ is taken modulo $(2r+1)$ for each $1\leq j \leq 2r+1$.
\end{proposition}

Since $\dim(T)\geq 0$ for every sectionable tournament $T$, the following holds.

\begin{corollary}\label{cor:dim}
If $\dim(P^i)=d_i$ for each $i\in [2r+1]$ in $T=R_{2r+1}(P^1,P^2,\ldots,P^{2r+1})$, then  
\[ 
\dim(T)\geq \max_{\underset{ i\neq j}{ i,j \in \{1,2,\ldots,2r+1\}}} \{d_i+d_{j}+r \}.
\]
\end{corollary}

\begin{definition}
We call a set of ordered integers $I:=(i_0,i_1,\ldots,i_j)$ as a \emph{block sequence} in $T=R_m(P^1,\ldots,P^m)$ with respect to a sectionable subtournament $S_j$ of $T$ if there exist sectionable tournaments $S_0,S_1,\ldots,S_{j-1}$ such that $S_r$ is the $(i_r)^{th}$ block of $S_{r-1}$ for each $1\leq r\leq j$, where $S_0$ is the $(i_0)^{th}$ block of $T$. If $I=(i_0,i_1,\ldots,i_j)$ is a block sequence in $T$ with respect to $S_j$, and $Q$ is a sectionable tournament, we denote by $T(I;Q)$, the tournament obtained from $T$ by replacing $S_j$ with $Q$. 
\end{definition}

For an example, consider the tournament $T'=R_5(R_1,T,R_3,R_5,R_1)$, where $T$ is the tournament depicted in Figure~\ref{fig.ex-acy}. %Note that the longest block sequence occurring in a highly regular tournament $R_m$ is of length $2$. It then follows that  the longest sequence in  $T'$ is of length $4$. 
We observe that the tournament $R_3$ in $T$  consisting of the vertices $3,4,5$ can be detected via a block sequence $I=(2,3)$, where $S_0=T$, $S_1=R_3$. Indeed, $T'(I;R_1)$ is a tournament obtained from $T'$ by replacing the subblock $R_3$ of $T$ with $R_1$ so that we have $T'(I;R_1)=R_5(R_1,T^*,R_3,R_5,R_1)$, where  $T^*=R_5(R_1,R_1,R_1,R_3,R_1)$.

\begin{lemma}\label{lem:contraction Tr}
Let $I=(i_0,i_1,\ldots,i_j)$ be a block sequence in a sectionable tournament $T=R_m(P^1,P^2,\ldots,P^m)$ with respect to a sectionable subtournament $S_j$ of $T$. If $S_j$ is isomorphic to a highly regular tournament $R_s$ for some $s\geq 5$, then $\Acy(T)\simeq \Acy(T(I;R_3))$.
\end{lemma}

\begin{proof}
We first prove the claim when $j=0$. By using the symmetry of highly regular tournament, we may assume without loss of generality that $S_0=P^1$. Furthermore, it is sufficient to verify that $\CS(\Acy(T))= \CS(\Acy(T(I; R_3)))$ \cite[Theorem 4.4]{jacop}.  The main idea of our proof is based on creating an acyclic matching on $\Acy(T)$ which allows us to divide $\CS(\Acy(T))$ into two disjoint subsets. One of those will contribute only a critical cell of dimension $1$ to $\Acy(T)$. On the other hand, the contribution of the other set to $\Acy(T)$ is the same as with a tournament where  $S_0$ is replaced with $R_1$ in $T$.

For the latter claim, we follow the proof of Lemma \ref{lem:highly-reg-S1}.  Suppose that $V(P^1)=V(R_{s})=\{v_1,v_2,\ldots,v_{s}\}$, where $s=2t+1$ for some $t\geq 2$. We will construct an acyclic matching on $\Acy(T)$. First, select $v_1$ as a pivot, and let $\M_1$ be the set of pairs $(\gamma,\gamma\cup \{v_1\})$ for $\gamma \in \Acy(T)$ with $v_1 \notin \gamma$. It follows that  the set of unmatched elements with respect to $\M_1$ is \linebreak $\DE(v_1)=\langle \alpha_1,\alpha_2,\ldots,\alpha_k,\beta_1,\beta_2,\ldots,\beta_\ell \rangle_T$, where $\alpha_i\subset V(P^1)$ and $\beta_j\subset V(T-P^1)$. By Lemma \ref{lem:highly-reg-S1}, we have 
\begin{align*}
\{\alpha_1,\alpha_2,\ldots,\alpha_k\}=\{v_2v_{t+2},v_3v_{t+2},&v_3v_{t+3},v_4v_{t+2},v_4v_{t+3},v_4v_{t+4},\ldots,\\
&v_{t+1}v_{t+2}, v_{t+1}v_{t+3},v_{t+1}v_{t+4},\cdots,v_{t+1}v_{2t+1}\}. 
\end{align*} 
Observe that that when $i\in [k]$ and $j\in [\ell]$, the set $\alpha_i\cup\beta_j$ does not create a face in $\Acy(T)$, and so such a set cannot be contained in $\DE(v_1)$. It then follows that $\DE(v_1)$ can be partitioned into two sets $\DE_1=\langle \alpha_1,\alpha_2,\ldots,\alpha_k \rangle_T$ and $\DE_2=\langle\beta_1,\beta_2,\ldots,\beta_\ell\rangle_T$. This also implies that $\CS(\Acy(T))=\CS(\DE_1) \cup \CS(\DE_2)$.

We next consider $\DE_1$ and construct a Morse matching on it as in Lemma \ref{lem:highly-reg-S1}. Select $v_{t+1}$ as a pivot on the family $\DE_1$. By the structure of highly regular tournaments, any set of ordered $t+1$ elements forms a facet. So, for $\gamma \in \DE_1$ with $v_{t+1} \notin \gamma$, if $\{v_i,v_j\} \subset \gamma$ for $2 \leq i \leq t$ and $t+2\leq j \leq 2t$, then the set of pairs $\big( \gamma, \gamma \cup \{v_{t+1}\}  \big)$ forms an acyclic matching, say $\M_2$. In this case, the unmatched elements with respect to $\M_2$ can be given as follows.
$$\U=\DE_1 \setminus \{\gamma \in \DE_1 : \{v_i,v_j\} \subset \gamma \text{ for }  2 \leq i \leq t\;\textnormal{and}\; t+2\leq j \leq 2t \}$$
This implies that $ \beta \in \U$ if and only if $\{v_{t+1},v_k\} \subset \beta$ for  $ t+2\leq k \leq 2t+1$ and $\{v_i,v_j \} \not \subset \beta$ for each $2\leq i \leq t$ and $t+1\leq j \leq 2t$. \medskip

Let $v_{t+2}$ be the last pivot on the family $\U$. Similar to the above cases,  it is clear that $\beta \cup \{v_{t+2}\}$ is a face for every $\beta \in \U$. Therefore, if $\beta \in \U$ with $v_{t+2} \notin \beta$, we conclude that the set of pairs $\big( \beta, \beta \cup \{v_{t+2}\}  \big)$  forms an acyclic matching, say $\M_3$. Notice that the only remaining unmatched element is  $\{v_{t+1},v_{t+2}\}$. We thus obtain an acyclic matching $\M=\M_1\cup \M_2 \cup \M_3$ by Lemma \ref{lem:morse}. 
It follows $\CS(\DE_1)\equiv \{\{v_{t+1},v_{t+2}\}\}$ and consequently, $\CS(\Acy(T))= \{\{v_{t+1},v_{t+2}\}\} \cup \CS(\DE_2)$.

The above argument yields that  $\CS(\Acy(T(I; R_3)))= \{\{v_{t+1},v_{t+2}\}\} \cup \CS(\DE_2)$ in which we particularly choose $V(R_3)=\{v_1,v_{t+1},v_{t+2}\}$. This in turn implies that we have $\CS(\Acy(T))=\CS(\Acy(T(I; R_3)))$, which completes the proof when $j=0$.

Finally, when $j\geq 1$, we repeat the same argument as above by replacing the role of $P^1$ with $S_j$. 
\end{proof}

\begin{remark}\label{rem:Rs-R3}
In the view of Lemma \ref{lem:contraction Tr}, if a subblock of a sectionable tournament $T$ is a highly regular tournament of order at least five, then we may replace it by a directed triangle $C_3$ without altering the homotopy type of $\Acy(T)$. Therefore, we will implicitly assume that every non-trivial and highly regular tournament subblock of a sectionable tournament is a directed triangle $C_3$ in the rest of this section. On the other hand, if a subblock of a sectionable tournament $T$ is itself a transitive tournament, we may think of it as $R_1$ by repeated use of Theorem~\ref{thm:hmtpy-edge}.
\end{remark}

Following the explanation given in Remark~\ref{rem:Rs-R3}, we may consider  any subblock which is a highly regular tournament as a copy of $R_3$.
\begin{definition}
Let $T=R_m(P^1,P^2,\ldots,P^m)$ be a sectionable tournament. If $I=(i_0,i_1,\ldots,i_j)$ is a block sequence in $T$ with respect to a sectionable subtournament $S_j=R_3$ of $T$, we then call $V(S_j)=\{v,w,z\}$ as a \emph{deep-triangle} of $T$ with the block sequence $I$.
\end{definition}

We remark that if $T=R_m(P^1,P^2,\ldots,P^m)$ is a sectionable tournament and $\{v,w,z\}$ is a deep-triangle in $T$, then  $\DE(v)=\langle wz,\beta_1,\beta_2,\ldots,\beta_k \rangle $ for some $\beta_i\in \Acy(T)$ where $\vert  \beta_i \vert=2$. 

\begin{lemma}\label{lem:cs-reduction}
Let $T=R_m(P^1,P^2,\ldots,P^m)$ be a sectionable tournament such that there exists a vertex $v \in V(P^1)$ which lies in a deep-triangle $\{v,w,z\}$ with a block sequence $I$, then 
$$\CS(\Acy(T))=\CS(\Sigma(wz:v)) \cup \CS(\Acy(T(I; R_1))) .$$
\end{lemma}

\begin{proof}
Since we have  $\DE(v)=\langle wz,\beta_1,\beta_2,\ldots,\beta_k \rangle $, after removing all elements of pairs $(\gamma,\gamma\cup \{v\})$ for the pivot $v$, the set of unmatched elements is 
$$ U= \Sigma(wz:v) \cup \bigcup_{i=1}^k \Sigma(\beta_i:v,wz,\beta_1,\ldots,\beta_{i-1}) $$
Note that for each $\gamma \in U$, 
\begin{itemize}
\item if $\{w,z\} \subset \gamma$, then $\beta_i \nsubseteq \gamma$, since $\{w,z\}\cup \beta_i$ is a non-face.
\item if $\{w,z\} \nsubseteq \gamma$, then $w \notin \gamma$ and $z \notin \gamma $, since $ \gamma$ contains some  $\beta_i$ as a subset such that $\{w\}\cup \beta_i$ and $\{z\}\cup\beta_i$ are non-face. 
\end{itemize}
So, for every $\gamma \in \Sigma(wz:v)$ and $\alpha\in U-\Sigma(wz:v)$, we conclude that $\vert \gamma \setminus \alpha \vert \geq 2$. This in turn forces that $\{\gamma,\alpha\}$ can not be a matched pair with respect to any pivot. Therefore, we may write $\CS(\Acy(T))=\CS(\Sigma(wz:v)) \cup \CS(\bigcup_{i=1}^k \Sigma(\beta_i:v,wz,\beta_1,\ldots,\beta_{i-1}))$. \medskip

Since none of $\gamma \in U-\Sigma(wz:v) $ contains any of the vertices in $\{v,w,z\}$, we have that  $$\CS(\bigcup_{i=1}^k \Sigma(\beta_i:v,wz,\beta_1,\ldots,\beta_{i-1}))=\CS(\Acy(T(I; R_1))).$$ 

Indeed, $wz$ is an unsaturated edge in $T-v$  so that $\Acy(T)\simeq\Acy(T) \sslash wz \simeq \Acy(T(I;R_1))$ by Theorem \ref{thm:hmtpy-edge}. Thus, $\CS(\Acy(T))=\CS(\Sigma(wz:v)) \cup \CS(\Acy(T(I;R_1)))$ as claimed. 
\end{proof}

\begin{lemma}\label{lem:cs-sigma-T}
Let $T=R_m(P^1,P^2,\ldots,P^m)$ be a sectionable tournament with $m=2r+1$, and $\CS(\Acy(P^{i}))$ is determined by an acyclic matching $\M_i$ via a pivot set $( q^{i}_1,q^{i}_2,\ldots,q^{i}_{t_i} )$, then 
\begin{align*}
 \CS(\Sigma(wz:q^1_1)_T) & \equiv \CS(\Sigma(wz:q^1_1)_{P^1}) \cupdot \\
& \Bigg(  \bigg[   \Big( \CS( \Acy(P^{r+2})) \cupdot \CS( \Acy(P^{r+3}) ) \cupdot  \ldots  \cupdot \CS(\Acy(P^{2r+1})) \Big) \cup  \\
& \Big( \CS(\Acy(P^{r+3}))  \cupdot  \ldots \cupdot \CS(\Acy(P^{2r+1})) \cupdot \CS(\Acy(P^{2}))\Big)  \cup
\cdots \\
& \cup \Big( \CS(\Acy(P^{2})) \cupdot  \CS(\Acy(P^{3})) \cupdot  \ldots \cupdot \CS(\Acy(P^{r+1}))\Big)  \bigg] 
\cup \\
& \bigg[  \Big(   \CS(\Acy(P^{r+3})) \cupdot  \ldots \cupdot \CS(\Acy(P^{2r+1})) \cupdot \{\{q^2_1\}\}  \Big) \cup  \\
& \Big(    \CS(\Acy(P^{r+4}) ) \cupdot  \ldots \cupdot \CS(\Acy(P^{2r+1})) \cupdot \CS(\Acy(P^{2})) \cupdot \{\{q^3_1\}\} \Big) \cup 
\cdots \\
& \cup \Big(  \CS(\Acy(P^{2}) ) \cupdot  \ldots \cupdot \CS(\Acy(P^{r})) \cupdot \{\{q^{r+1}_1\}\}  \Big)  \bigg] 
\Bigg),
\end{align*}
where $wz\in \DE(q_{1_1})$.
\end{lemma}
\begin{proof}
Let $\alpha \in \Sigma(wz:q^1_1)_T$ be  an element. Since every (cyclically) ordered set of $(r+1)$ consecutive vertices on a highly regular tournament $R_{2r+1}$ creates a face, the set $\alpha$ is clearly a subset of $\A$, where
\begin{align*}
 \mathcal{A}=  &\Big( \Acy(P^{r+2}) \cupdot  \Acy(P^{r+3}) \cupdot  \ldots\cupdot \Acy(P^{2r+1}) \cupdot \Sigma(wz:q^1_1)_{P^1}\Big) \cup \\
&\Big( \Acy(P^{r+3}) \cupdot  \Acy(P^{r+4}) \cupdot  \ldots \cupdot  \Acy(P^{2r+1}) \cupdot \Sigma(wz:q^1_1)_{P^1} \cupdot  \Acy(P^{2})\Big)  \cup \cdots \\ 
&\cup \Big( \Sigma(wz:q^1_1)_{P^1} \cupdot   \Acy(P^{2}) \cupdot  \ldots \cupdot \Acy(P^{r+1}) \Big).
\end{align*}

Therefore, 
\begin{align*}
 \Sigma(wz:q^1_1)_T & \equiv \Sigma(wz:q^1_1)_{P^1} \cupdot \bigg[  \Big( \Acy(P^{r+2}) \cupdot \Acy(P^{r+3}) \cupdot   \ldots \cupdot \Acy(P^{2r+1}) \Big) \cup  \\
& \Big( \Acy(P^{r+3}) \cupdot  \Acy(P^{r+4}) \cupdot  \ldots \cupdot \Acy(P^{2r+1}) \cupdot \Acy(P^{2})\Big)  
\cdots \\
& \Big( \Acy(P^{2}) \cupdot  \Acy(P^{3}) \cupdot  \ldots \cupdot \Acy(P^{r+1})\Big)  \bigg]. 
\end{align*}

Now, recall that each $\CS(\Acy(P^{i}))=\CS(\Sigma(\emptyset:\emptyset)_{P^{i}})$ is determined by an acyclic matching $\M_i$ via the pivot set $( q^i_1,q^i_2,\ldots,q^i_{t_i} )$. We will find an acyclic matching for the right hand side of the above equivalence.

We write the set of pivots of $\CS(\Acy(P^{r+2})),\CS(\Acy(P^{r+3})),\ldots,\CS(\Acy(P^{r+1}))$ on the family $ \Sigma(wz:q^1_1)_T$ as
$$( q^{r+2}_1,q^{r+2}_2,\ldots,q^{r+2}_{t_{r+2}} ), ( q^{r+3}_1,q^{r+3}_2,\ldots,q^{r+3}_{t_{r+3}} ), \ldots, ( q^{r+1}_1,q^{r+1}_2,\ldots,q^{r+1}_{t_{r+1}} )$$
respectively. After applying these sets consecutively, the resulting created matching is acyclic by Lemma \ref{lem:morse}. We claim that the unmatched elements are $\U\cup \V$, where

\begin{align*}
\U=&\Big( \CS(\Acy(P^{r+2})) \cupdot  \CS(\Acy(P^{r+3})) \cupdot  \ldots \cupdot  \CS(\Acy(P^{2r+1})) \cupdot \CS(\Sigma(wz:q^1_1)_{P^1})\Big) \cup \\
&\Big( \CS(\Acy(P^{r+3})) \cupdot  \CS(\Acy(P^{r+4})) \cupdot  \ldots \cupdot  \CS(\Sigma(wz:q^1_1)_{P^1})\cupdot  \CS(\Acy(P^{2}))\Big)\cup \ldots \\
&\cup \Big( \CS(\Sigma(wz:q^1_1)_{P^1}) \cup \CS(\Acy(P^{2})) \cupdot  \CS(\Acy(P^{3})) \cupdot  \ldots \cupdot \CS(\Acy(P^{r+1}))\Big) \\
\end{align*} 
and
\begin{align*}
\V=&\Big( \CS(\Acy(P^{r+3}))  \cupdot  \ldots  \cupdot \CS(\Acy(P^{2r+1}))  \cupdot \CS(\Sigma(wz:q^1_1)_{P^1}) \cupdot \{\{q^2_1\}\} \Big) \cup \\
&\Big( \CS(\Acy(P^{r+4})) \cupdot   \ldots  \cupdot  \CS(\Sigma(wz:q^1_1)_{P^1})\cupdot  \CS(\Acy(P^{2}))   \cupdot \{\{q^3_1\}\} \Big)\cup \ldots \\
&\cup \Big( \CS(\Sigma(wz:q^1_1)_{P^1}) \cupdot \CS(\Acy(P^{2})) \cupdot  \ldots \cupdot \CS(\Acy(P^{r})) \cupdot \{\{q^{r+1}_1\}\} \Big). 
\end{align*}

In fact, each element $\alpha \in \Acy(P^{r+3}) \cupdot \Acy(P^{r+4}) \cupdot  \ldots \cupdot \Acy(P^{2r+1})\cupdot \CS(\Sigma(wz:q^1_1)_{P^1})$ is matched with $\alpha\cup \{q^{r+3}_1\}$, and the other sets can be treated similarly. It follows that any ordered set of $(r+1)$ elements consisting of $\CS(\Acy(P^{i_1})),\ldots, \CS(\Acy(P^{i_{r+1}}))$ for  some consecutive $i_1\leq \ldots\leq i_{r+1}$ is unmatched. Notice that if one of $P^i$ consists of a unique vertex, then $\CS(\Acy(P^{i}))$ would be $\emptyset$; hence, any ordered set of $(r+1)$ elements containing $\CS(\Acy(P^{i}))$ would be $\emptyset$ by the definition of dot-join operation.  Therefore, $\U$ is the set of unmatched elements. \medskip

On the other hand, consider an element $\alpha$ which consists of ordered set of $r$ elements containing an element of  $\CS(\Sigma(wz:q^1_1)_{P^1})$, say $\alpha \in \Acy(P^{r+3}) \cupdot \Acy(P^{r+4}) \cupdot  \ldots \cupdot \Acy(P^{2r+1})\cupdot \CS(\Sigma(wz:q^1_1)_{P^1})$.  Then $\alpha\cup \{q^2_1\}$ is unmatched; hence, $\V$ is the set of unmatched elements. Thus, we can define a partial Morse matching on such decomposition so that the combination of them is a Morse matching by Lemma \ref{lem:morse}.

Consequently,  $\CS(\Sigma(wz:q^1_1)_T)  \equiv \U \cup \V$. Thus, we obtain the following equivalence:
 \begin{align*}
 \CS(\Sigma(wz:q^1_1)_T) & \equiv \CS(\Sigma(wz:q^1_1)_{P^1}) \cupdot \\
 &\Bigg(  \bigg[   \Big( \CS( \Acy(P^{r+2})) \cupdot \CS( \Acy(P^{r+3}) ) \cupdot  \ldots \cupdot \CS(\Acy(P^{2r+1})) \Big) \cup  \\
& \Big( \CS(\Acy(P^{r+3}))  \cupdot  \ldots \cupdot \CS(\Acy(P^{2r+1})) \cupdot \CS(\Acy(P^{2}))\Big)  \cup
\cdots \\
& \cup \Big( \CS(\Acy(P^{2})) \cupdot  \CS(\Acy(P^{3})) \cupdot  \ldots \cupdot \CS(\Acy(P^{r+1}))\Big)  \bigg] 
\cup \\
& \bigg[  \Big(   \CS(\Acy(P^{r+3})) \cupdot  \ldots \cupdot \CS(\Acy(P^{2r+1})) \cupdot \{\{q^2_1\}\}  \Big) \cup  \\
& \Big(    \CS(\Acy(P^{r+4}) ) \cupdot  \ldots \cupdot \CS(\Acy(P^{2r+1})) \cupdot \CS(\Acy(P^{2})) \cupdot \{\{q^3_1\}\} \Big) \cup 
\cdots \\
& \cup \Big(  \CS(\Acy(P^{2}) ) \cupdot  \ldots \cupdot \CS(\Acy(P^{r})) \cupdot \{\{q^{r+1}_1\}\}  \Big)  \bigg] 
\Bigg).
\end{align*}

%where we use $q_{(i)_j}$ instead of $q_{i_j}$. 
\end{proof}

\begin{theorem}\label{thm:cs-tour}
Let $T=R_m(P^1,P^2,\ldots,P^m)$ be a sectionable tournament with $m=2r+1$, and suppose that $\CS(\Acy(P^{j}))$ is determined by an acyclic matching $\M_j$ via the pivot set $( q^j_1,q^j_2,\ldots,q^j_{t_j} )$ for $j\in [m]$, then
\begin{align*}
\CS(\Acy&(T)) \equiv  \{\{q^1_1,q^2_1\}\} \cup  \\
&\bigcup_{ i \in \{1,2,\ldots,m\}}  \Bigg[ \CS(\Acy(P^{i}))\cupdot \ldots \cupdot \CS(\Acy(P^{i+r-1})) \cupdot \bigg( \{\{q^{i+r}_1 \}\} \cup \CS(\Acy(P^{i+r})) \bigg)  \Bigg], 
\end{align*}
where indices of $P^{j}$'s are taken modulo $m$.
\end{theorem}

\begin{proof}
Consider $v=q^1_1\in V(P^1)$ as a pivot, then by Lemma \ref{lem:contraction Tr} the set of unmatched elements is 
 $$\DE(v)=\langle wz,\beta_1,\beta_2,\ldots,\beta_k \rangle $$
where  $\{v,w,z\}$ is a deep-triangle corresponding to a block sequence $I=(i_0,i_1,\ldots,i_j)$ with respect to $S_j=\{v,w,z\}$, and  $\{\beta_1,\beta_2,\ldots,\beta_k\}\subset E(T-S_j)$. Note that $\CS(\Acy(T))=\CS(\Sigma(wz:v)) \cup \CS(\Acy(T(I; R_1))) $  by Lemma \ref{lem:cs-reduction}. We apply to an induction on $n=\vert V(T) \vert$. Since the base case is trivial, we suppose for the induction step that the claim is true for sectionable tournaments of size at most $n-1$. Then,
\begin{align*}
\CS(\Acy(T))  \equiv & \CS(\Sigma(wz:q^1_1)_T) \cup \CS(\Acy(T(I; R_1)))  \\
 \equiv & \CS(\Sigma(wz:q^1_1)_T) \cup  \{\{q^1_1,q^2_1\}\} \cup \\ 
& \bigcup_{{ i \in \{1,\ldots,m\}}}   \bigg[ \CS(\Acy(P^{i}))\cupdot \ldots\cupdot \CS(\Acy(P^{i+r-1})) \cupdot \\
& \bigg( \{\{q^{i+r}_1 \}\} \cup \CS(\Acy(P^{i+r})) \bigg)   \bigg]_{T(I; R_1)}  \\  
 \equiv & \CS(\Sigma(wz:q^1_1)_T) \cup  \{\{q^1_1,q^2_1\}\} \cup \\ 
& \bigcup_{ i \in \{1,2,\ldots,m\}}   \bigg[ \CS(\Acy(P^{i}))\cupdot \ldots\cupdot \CS(\Acy(P^{i+r-1})) \cupdot  \{\{q^{i+r}_1 \}\}   \bigg]_{T(I; R_1)}    \\ 
& \bigcup_{{ i \in \{1,2,\ldots,m\}}}   \bigg[ \CS(\Acy(P^{i}))\cupdot \ldots\cupdot \CS(\Acy(P^{i+r-1})) \cupdot  \CS(\Acy(P^{i+r}))   \bigg]_{T(I; R_1)}   \\ 
 \equiv & \CS(\Sigma(wz:q^1_1)_T) \cup  \{\{q^1_1,q^2_1\}\} \cup \\ 
& \bigcup_{{ i \in \{2,3,\ldots,r+2\}}}   \bigg[ \CS(\Acy(P^{i}))\cupdot \ldots\cupdot \CS(\Acy(P^{i+r-1})) \cupdot  \{\{q^{i+r}_1 \}\}      \bigg]  \cup \\ 
& \bigcup_{{ i \in \{2,3,\ldots,r+1\}}}   \bigg[ \CS(\Acy(P^{i}))\cupdot \ldots\cupdot \CS(\Acy(P^{i+r-1})) \cupdot \CS(\Acy(P^{i+r}))     \bigg]  \cup \\ 
& \bigcup_{{ i \in \{r+3,r+4,\ldots,m,1\}}}   \bigg[  \CS(\Acy(P^{i}))\cupdot \ldots\cupdot \CS(\Acy(P^{i+r-1})) \cupdot  \{\{q_{(i+r)_1}\} \}  \bigg] _{T(I; R_1)}  \cup \\  
& \bigcup_{{ i \in \{r+2,r+3,\ldots,m,1\}}}   \bigg[  \CS(\Acy(P^i))  \cupdot \ldots\cupdot \CS(\Acy(P^{i+r-1}))\cupdot \CS(\Acy(P^{i+r}))  \bigg]_{T(I; R_1)}.  
\end{align*}

%Recall that $\Sigma(\emptyset:\emptyset)_P\equiv \Acy(P)$ for any tournament $P$. Then, 
We obtain the following equivalence by applying Lemma \ref{lem:cs-sigma-T} to the above equivalence. 
\begin{align*}
\CS(\Acy(T)) \equiv &      \CS(\Sigma(wz:q^1_1)_{P^1}) \cupdot \\
&\Bigg(  \bigg[   \Big( \CS( \Acy(P^{r+2})) \cupdot \CS( \Acy(P^{r+3}) ) \cupdot  \ldots \cupdot \CS(\Acy(P^{2r+1})) \Big) \cup  \\
& \Big( \CS(\Acy(P^{r+3}))  \cupdot  \ldots \cupdot \CS(\Acy(P^{2r+1})) \cupdot \CS(\Acy(P^{2}))\Big)  \cup
\cdots \\
& \cup \Big( \CS(\Acy(P^{2})) \cupdot  \CS(\Acy(P^{3})) \cupdot  \ldots \cupdot \CS(\Acy(P^{r+1}))\Big)  \bigg] 
\cup \\
& \bigg[  \Big(   \CS(\Acy(P^{r+3})) \cupdot  \ldots \cupdot \CS(\Acy(P^{2r+1})) \cupdot \{\{q^2_1\}\}  \Big) \cup  \\
& \Big(    \CS(\Acy(P^{r+4}) ) \cupdot  \ldots \cupdot \CS(\Acy(P^{2r+1})) \cupdot \CS(\Acy(P^{2})) \cupdot \{\{q^3_1\}\} \Big) \cup 
\cdots \\
& \cup \Big(  \CS(\Acy(P^{2}) ) \cupdot  \ldots \cupdot \CS(\Acy(P^{r})) \cupdot \{\{q^{r+1}_1\}\}  \Big)  \bigg] 
\Bigg)       
\cup  \{\{q_{1_1},q_{2_1}\}\} \cup \\ 
& \bigcup_{{ i \in \{2,3,\ldots,r+2\}}}   \bigg[ \CS(\Acy(P^{i}))\cupdot \ldots\cupdot \CS(\Acy(P^{i+r-1})) \cupdot  \{\{q^{i+r}_1 \}\}      \bigg]  \cup \\ 
& \bigcup_{{ i \in \{2,3,\ldots,r+1\}}}   \bigg[ \CS(\Acy(P^{i}))\cupdot \ldots\cupdot \CS(\Acy(P^{i+r-1})) \cupdot \CS(\Acy(P^{i+r}))     \bigg]  \cup \\ 
& \bigcup_{{ i \in \{r+3,r+4,\ldots,m,1\}}}   \bigg[  \CS(\Acy(P^{i}))\cupdot \ldots\cupdot \CS(\Acy(P^{i+r-1})) \cupdot  \{\{q^{i+r}_1\} \}  \bigg]_{T(I; R_1)}  \cup \\  
& \bigcup_{{ i \in \{r+2,r+3,\ldots,m,1\}}}   \bigg[  \CS(\Acy(P^i))  \cupdot \ldots\cupdot \CS(\Acy(P^{i+r-1}))\cupdot \CS(\Acy(P^{i+r}))  \bigg]_{T(I; R_1)}   \\
 \equiv &     \bigg(  \CS(\Sigma(wz:q^1_1)_{P^1}) \cup \CS(\Acy(P^1))_{T(I; R_1)}  \bigg) \cupdot \\
 & \Bigg(  \bigg[ \CS( \Acy(P^{r+2})) \cupdot \CS( \Acy(P^{r+3}) ) \cupdot  \ldots \cupdot \CS(\Acy(P^{2r+1})) \Big) \cup  \\
& \Big( \CS(\Acy(P^{r+3}))  \cupdot  \ldots \cupdot \CS(\Acy(P^{2r+1})) \cupdot \CS(\Acy(P^{2}))\Big)  \cup
\cdots \\
& \cup \Big( \CS(\Acy(P^{2})) \cupdot  \CS(\Acy(P^{3})) \cupdot  \ldots \cupdot \CS(\Acy(P^{r+1}))\Big)  \bigg] 
\cup \\
& \bigg[  \Big(   \CS(\Acy(P^{r+3})) \cupdot  \ldots \cupdot \CS(\Acy(P^{2r+1})) \cupdot \{\{q^2_1\}\}  \Big) \cup  \\
& \Big(    \CS(\Acy(P^{r+4}) ) \cupdot  \ldots \cupdot \CS(\Acy(P^{2r+1})) \cupdot \CS(\Acy(P^{2})) \cupdot \{\{q^3_1\}\} \Big) \cup 
\cdots \\
& \cup \Big(  \CS(\Acy(P^{2}) ) \cupdot  \ldots \cupdot \CS(\Acy(P^{r})) \cupdot \{\{q^{r+1}_1\}\}\Big)  \bigg] 
\Bigg)       \cup  \{\{q^1_1,q^2_1\}\} \cup \\ 
& \bigcup_{{ i \in \{2,3,\ldots,r+2\}}}   \bigg[ \CS(\Acy(P^{i}))\cupdot \ldots\cupdot \CS(\Acy(P^{i+r-1})) \cupdot  \{\{q^{i+r}_1 \}\}      \bigg]  \cup \\ 
& \bigcup_{{ i \in \{2,3,\ldots,r+1\}}}   \bigg[ \CS(\Acy(P^{i}))\cupdot \ldots\cupdot \CS(\Acy(P^{i+r-1})) \cupdot \CS(\Acy(P^{i+r}))     \bigg].  
\end{align*}
By Lemma \ref{lem:cs-reduction} together with Proposition \ref{prop:disjoint ground}, we obtain the following
\begin{align*}
\CS(\Acy(T))  \equiv & \{\{q^1_1,q^2_1\}\} \cup  \\
&\bigcup_{{ i \in \{1,\ldots,m\}}}  \Bigg[ \CS(\Acy(P^{i}))\cupdot \ldots\cupdot \CS(\Acy(P^{i+r-1})) \cupdot \bigg( \{\{q^{i+r}_1 \}\} \cup \CS(\Acy(P^{i+r})) \bigg)  \Bigg], 
\end{align*}
which completes the proof.
\end{proof}

\begin{definition}
Let $T=R_m(P^1,P^2,\ldots,P^m)$ be a sectionable tournament. We call the highest dimension of a critical cell occurring with respect to an acyclic matching $\M$ of $\Acy(T)$ as the \emph{depth}
of $T$, denoted by $\depth_{\M}(T)$.  In particular, we omit the subscript $\M$ whenever it is clear from the context. 
\end{definition}

In the rest of this section,  the notation  $\depth(T)$ for a sectionable tournament $T$ refers to the matchings as constructed in Lemmas \ref{lem:highly-reg-S1}, \ref{lem:contraction Tr}, \ref{lem:cs-reduction}, \ref{lem:cs-sigma-T}, Theorems \ref{thm:cs-tour} and \ref{thm:depth}.

We call a sequence $d_1,\ldots,d_k$ of non-negative integers a \emph{non-zero sequence} if $d_i\geq 1$ for each $i\in [k]$, and otherwise, it is said to be \emph{trivial sequence}.

\begin{theorem}\label{thm:depth}
Let $T=R_m(P^1,P^2,\ldots,P^m)$ be a sectionable tournament with $m=2r+1$ such that $\depth(P^i)=d_i$ for each $i\in [m]$. Then,
\[ 
\depth(T)=\max_{{ i \in \{1,\ldots,m\}}} \{1,d_i+d_{i+1}+\ldots+d_{i+r}+r: \textnormal{either $d_{i},\ldots,d_{i+r-1}$  or  $d_{i+1},\ldots,d_{i+r}$ are non-zero}\}
\]
where the index of $d_{j}$ is taken modulo $m$ for each $1\leq j \leq 2r+1$.
\end{theorem}
\begin{proof}
By Theorem \ref{thm:cs-tour}, we have the following equivalence
\begin{align*}
\CS(\Acy(T))  \equiv & \{\{q^1_1,q^2_1\}\} \cup  \\
&\bigcup_{{ i \in \{1,\ldots,m\}}}  \Bigg[ \CS(\Acy(P^{i}))\cupdot \ldots\cupdot \CS(\Acy(P^{i+r-1})) \cupdot \bigg( \{\{q^{i+r}_1 \}\} \cup \CS(\Acy(P^{i+r})) \bigg)  \Bigg]. 
\end{align*}
Recall that if one of two families of sets $\A$ and $\B$ is trivial, i.e., it is an empty set, then $\A \cupdot \B = \emptyset $ by  the definition of dot-join operation.  It follows that the size of the highest dimension of the critical cells occurring in $\CS(\Acy(T))$ comes from  either  consecutive $r+1$ non-zero $d_i$'s  or  consecutive $r$ non-zero $d_i$'s. 
On the other hand, if $\depth(P^i)=d_i$ and $\depth(P^j)=d_j$, then for some Morse matching, there exist critical cells $\alpha=\{c_1,c_2\ldots,c_{d_i+1}\}$  and $\beta=\{f_1,f_2\ldots,f_{d_j+1}\}$ in $\CS(\Acy(P^{i}))$ and $\CS(\Acy(P^{j}))$ respectively.  
Since  $\{c_1,c_2\ldots,c_{d_i+1}\} \uplus \{f_1,f_2\ldots,f_{d_j+1}\}=\{ c_1,c_2\ldots,c_{d_i+1}, f_1,f_2\ldots,f_{d_j+1}\}$ is a critical cell of size $d_i+d_j+1$ in $\CS(\Acy(P^{i})) \cupdot \CS(\Acy(P^{j}))$, the highest dimension of the critical cells in CW-decomposition of $\Acy(T)$ will be $d_i+d_{i+1}+\ldots+d_{i+r}+r$ for $i \in \{1,2,\ldots,m\}$ $(\mod (2r+1))$, when either $d_{i},d_{i+1},\ldots,d_{i+r-1}$  or  $d_{i+1},d_{i+2},\ldots,d_{i+r}$ are non-zero. 
\end{proof}

\begin{example}
For the sectionable tournament $T=R_5(R_1,R_1,R_3,R_3,R_1)$ depicted in Figure \ref{fig.ex-acy}, we have already obtained $\Acy(T)$ is homotopy equivalent to a CW-complex with one cell in each dimension $0,1$ and $4$ in Example \ref{ex:t5}. Alternatively, we can find the same result by applying to Theorem \ref{thm:cs-tour}.  We first let
\begin{align*}
\CS(\Acy(P^{1}))=&\CS(\Acy(R_{1}))=\emptyset \quad  \ \textnormal{ (via the pivot }  q^1_1=1) \\
\CS(\Acy(P^{2}))=&\CS(\Acy(R_{1}))=\emptyset \quad  \ \textnormal{ (via the pivot }  q^2_1=2) \\
\CS(\Acy(P^{3}))=&\CS(\Acy(R_{3}))=\{\{4,5\}\} \quad  \ \textnormal{ (via the pivot }  q^3_1=3) \\
\CS(\Acy(P^{4}))=&\CS(\Acy(R_{3}))=\{\{7,8\}\}  \quad  \ \textnormal{ (via the pivot }  q^4_1=6) \\
\CS(\Acy(P^{5}))=&\CS(\Acy(R_{1}))=\emptyset \quad  \ \textnormal{ (via the pivot }  q^5_1=9) 
\end{align*}
and therefore
\begin{align*}
\CS(\Acy(T))  \equiv & \{\{q^1_1,q^2_1\}\} \cup  \bigg( \CS(\Acy(P^{3}))\cupdot \CS(\Acy(P^{4})) \cupdot \{\{q^5_1 \}\}   \bigg) \\
              \equiv & \{\{1,2\}\} \cup  \bigg( \{\{4,5\}\}  \cupdot \{\{7,8\}\} \cupdot  \{\{9 \}\}   \bigg) \\
              \equiv & \big\{\{1,2\}\big\} \cup  \big\{\{ 4,5,7,8,9\} \big\} 
\end{align*}
Hence, $\Acy(T)$ is homotopy equivalent to a CW-complex with one cell in each of the dimensions $0, 1$ and $4$.\medskip
\end{example}

\begin{definition}
A sectionable tournament $T=R_m(P^1,P^2,\ldots,P^m)$, $T$ is called an \emph{elementary sectionable tournament} if $m=2r+1$ and there does not exist (cyclically) $i \in [m]$ such that tournaments $P^i,P^{i+1},\ldots,P^{i+r-1}$ are non-transitive.   
\end{definition}
Notice that every highly regular tournament is an elementary sectionable tournament.\medskip

For an elementary sectionable tournament $T$, we will show that $\Acy(T)\simeq S^1$ in Lemma \ref{lem:weakly-s1}. 

\begin{lemma}\label{lem:weakly-s1}
If $T=R_m(P^1,P^2,\ldots,P^m)$ is an elementary sectionable tournament, then $\Acy(T)\simeq S^1$.
\end{lemma}
\begin{proof}
Let $T=R_m(P^1,P^2,\ldots,P^m)$ be a sectionable tournament with $m=2r+1$ such that $\depth(P^i)=d_i$. By Theorem \ref{thm:cs-tour}, the highest dimension of critical cells in the CW-decomposition of $\Acy(T)$ comes from  either  consecutive $r+1$ non-zero $d_i$'s or  consecutive $r$ non-zero $d_i$'s.  Since there is no  (cyclically) $i \in [m]$ such that tournaments $P^i,P^{i+1},\ldots,P^{i+r-1}$ are non-transitive, the only critical cell in the corresponding CW-decomposition of $\Acy(T)$ is  $\{\{q_{1_1},q_{2_1}\}\}$ for some Morse matching, that is, $\Acy(T)\simeq S^1$.
%{\color{red!70!black}
%For the other direction, suppose that $\Acy(T)\simeq S^1$. Assume for a contradiction that $T$ is  a non-elementary sectionable tournament.  Then $\CS(\Acy(T))$ has more than one critical cell in dimension at least 1. If all those critical cells have the same dimension, then $\Acy(T)$ would be homotopy equivalent to a wedge of sphere by Corollary \ref{cor:Forman}, a contradiction. Otherwise, $\Acy(T)$ is homotopy equivalent to a CW-complex with at least one cell in dimension $k\geq 2$. Such a complex can not be homotopy equivalent to $S^1$. }
\end{proof}

\begin{definition}
Let $T=R_m(P^1,P^2,\ldots,P^m)$ be a sectionable tournament. We call the greatest integer $k$ as \emph{width} of $T$, denoted by $\width(T)$  if there exists $j\in [m]$ such that (cyclically) every block in the sequence $P^j,P^{j+1},\ldots,P^{j+k-1}$ is non-transitive.
\end{definition}

We remark that if a sectionable tournament $T=R_m(P^1,P^2,\ldots,P^m)$ is non-elementary, then $\width(T)\geq r$. On the contrary,  $\width(T)\leq r-1$ whenever $T$ is an elementary sectionable tournament. Furthermore, $\width(T)=0$ if and only if every block $P^i$ in $T=R_m(P^1,P^2,\ldots,P^m)$
is transitive. Obviously, if $T$ is a highly regular tournament, then $\width(T)=0$.

\begin{corollary}
Let $T=R_m(P^1,P^2,\ldots,P^m)$ be a sectionable tournament with $m=2r+1$. If $\width(T)\leq r-1$, then  $\Acy(T)\simeq S^1$.
\end{corollary}

The gap between $\depth(T)$ and $\dim(T)$ can be arbitrarily large for sectionable tournament $T$. Indeed, let us consider an elementary sectionable tournament $T=R_m(P^1,P^2,\ldots,P^m)$ with $|T|=n$ and $m=2r+1\geq 5$ such that $P^{2i-1}=R_1$ and $P^{2i}=R_3$ for each $i \in [r]$ and $P^m=R_1$. By Lemma \ref{lem:weakly-s1}, we have that $\Acy(T)\simeq S^1$ so that $\depth(T)=1$. On the other hand, we may calculate the dimension of $T$ by applying to Proposition \ref{prop:dim}.
$$\dim(T)=2\bigg\lceil\frac{r+1}{2}\bigg\rceil+\bigg\lfloor\frac{r+1}{2}\bigg\rfloor-1.$$

We next consider a non-elementary sectionable tournament $T=R_m(P^1,P^2,\ldots,P^m)$ with $|T|=n$ and $m=2r+1$ in which $P^1=R_3,,P^2=R_3,\ldots,P^{r+1}=R_3$, $P^{r+2}=R_1$,  $P^{r+3}=R_m(R_3,R_1,R_1,\ldots,R_1)$, and $P^{r+4}=R_1,P^{r+5}=R_1,\ldots,P^{2r+1}=R_1$. By Theorem \ref{thm:depth}, we have $$\depth(T)=\depth(P^1)+\depth(P^2)+\ldots+\depth(P^{r+1})+r=2r+1.$$ Moreover, the dimension of $\Acy(T)$ comes from the largest block $P^{r+3}$ because of $\dim(P^{r+3})=r+1$. By applying Proposition \ref{prop:dim}, we conclude that $$\dim(T)=\dim(P^3)+\dim(P^4)+\ldots+\dim(P^{r+3})+r=3r.$$

\begin{corollary}\label{cor:dept-dim-tri}
For every $k\geq1$, there exists a sectionable tournament $T$ such that $\dim(T)=\depth(T)+k$. 
\end{corollary}

In particular, $\dim(T)$ and $\depth(T)$ coincide for a $3$-sectionable tournament $T$ (compare to \cite{zd}).

\begin{corollary}
$\depth(T)=\dim(T)$ for every $3$-sectionable tournament $T$.
\end{corollary}

We finally characterize when $\dim(T)$ and $\depth(T)$ coincide in general. 

\begin{proposition}\label{prop:}
Let  $T=R_m(P^1,P^2,\ldots,P^m)$ be a sectionable tournament with $m=2r+1$ and let $\depth(P^i)=d_i$ for each $i\in [2r+1]$. Suppose $\dim(T)=d_j+d_{j+1}+\ldots+d_{j+r}+r$ for some $j\in [2r+1]$ $(\mod (2r+1))$. Then $\depth(T)=\dim(T)$ if and only if   $\min\{d_{j+1},d_{j+2},\ldots,d_{j+r-1}\}\geq 1$ and $\max\{d_{j},d_{j+r}\}\geq 1$.
\end{proposition}

\begin{proof}
Let  $T=R_m(P^1,P^2,\ldots,P^m)$ be a sectionable tournament such that $\dim(P^i)=d_i$ for each $i\in [2r+1]$. We suppose that $\dim(T)=d_j+d_{j+1}+\ldots+d_{j+r}+r$ for some $j\in [2r+1]$ $(\mod (2r+1))$. If $\depth(T)=\dim(T)$, then  either $d_{i},\ldots,d_{i+r-1}$  or  $d_{i+1},\ldots,d_{i+r}$ are non-zero by Theorem \ref{thm:depth}. This implies that $\min\{d_{j+1},d_{j+2},\ldots,d_{j+r-1}\}\geq 1$ and $\max\{d_{j},d_{j+r}\}\geq 1$.

Conversely, assume  $\min\{d_{j+1},d_{j+2},\ldots,d_{j+r-1}\}\geq 1$ and $\max\{d_{j},d_{j+r}\}\geq 1$. It follows that either $d_{i},\ldots,d_{i+r-1}$  or  $d_{i+1},\ldots,d_{i+r}$ are non-zero. Thus, $\depth(T)=\dim(T)$. 
\end{proof}

%%%%%%%%%%%%%%%%%%%%%%%%%%%%%%%%%%%%%%%%%%%%%%%%%%%%%%%%%%%%%%%%%%%%%%%%%%%%%%%%%%%%%%%%%
\section{Coloring of Sectionable Tournaments}
%%%%%%%%%%%%%%%%%%%%%%%%%%%%%%%%%%%%%%%%%%%%%%%%%%%%%%%%%%%%%%%%%%%%%%%%%%%%%%%%%%%%%%%%%
In this section, we construct an upper bound on the (acyclic) chromatic number of sectionable tournaments in terms of the dimension of their acyclic complexes as we promised. We begin with the following technical lemma.

\begin{proposition}\label{prop:2k}
$\ceil[\Big]{ \frac{k(2r+1)}{r+1}}=2k$ for every pair of positive integers $r\geq k$. 
\end{proposition}
\begin{proof}
We will use the facts $\ceil[\big]{ \frac{a}{b}}=\floor[\big]{ \frac{a+b-1}{b}}$ and $\floor[\big]{ \frac{a}{b}}=\ceil[\big]{ \frac{a-b+1}{b}}$ for $a,b\in \N$. Then,
\begin{align*}
\ceil[\Big]{ \frac{k(2r+1)}{r+1}}&=\floor[\Big]{ \frac{k(2r+1)+r}{r+1}}=\floor[\Big]{ \frac{k(r+1)+r(k+1)}{r+1}}\\
&=k+\floor[\Big]{ \frac{r(k+1)}{r+1}}=k+\ceil[\Big]{ \frac{r(k+1)-r}{r+1}}=k+\ceil[\Big]{ \frac{rk}{r+1}}=k+\ceil[\Big]{ \frac{rk+k-k}{r+1}}\\
&=k+\ceil[\Big]{ \frac{k(r+1)-k}{r+1}}=
k+k+\ceil[\Big]{ \frac{-k}{r+1}}=2k
\end{align*}
where we use $\ceil[\Big]{ \frac{-k}{r+1}}=0$, since $1\leq k\leq r$.
\end{proof}

\begin{lemma}\label{lem:maxcoloring-eq}
Let $T=R_{2r+1}(P^1,P^2,\ldots,P^{2r+1})$  be a tournament. If $\chi(P^i)=k$ for each $1\leq i\leq 2r+1$,  
then $\chi(T)\leq \ceil[\Big]{\frac{k(2r+1)}{r+1}}$.  
\end{lemma}
\begin{proof}
The idea of our proof is based on creating a palette of colors for each block which contains sufficiently many colors for it to produce a legal coloring for the tournament $T$. 
%{\color{blue} In more detail, consider a cyclicly ordered blocks $P^i,P^{i+1},\ldots,P^{i+r}$ in $T$ for $i\in [m]$, if we pick a transitive subset $S_j$ in $V(P^{i+j})$ for each $1\leq j \leq r$, then $K=S_1\cup S_2 \cup \ldots \cup S_{r+1}$ would be a transitive set in $T$. This means that all vertices in such a set $K$ can be assigned with the same color in a legal coloring for the tournament $T$. Moreover, if some $S_j$'s are non-transitive, then the set of vertices $K=S_1\cup S_2 \cup \ldots \cup S_{r+1}$ can be assigned with a color palette of size $\max \{\chi(T[S_j]) \ : \ S_j\subset K \}$ in a legal coloring for the tournament $T$.}

Let $\alpha=\lfloor\frac{k}{r+1}\rfloor$ and $\ell=k-\lfloor\frac{k}{r+1}\rfloor(r+1)$ for $k\in \N$. Note that $\ell$ is the remainder of $k$ divided by $r+1$. We have three cases.\medskip

\textbf{Case 1:} $\alpha=0$. Then $k=\ell\leq r$. Consider a cyclic orientation of $T$ and $k$-colorings of $P^1$ and $P^{r+2}$ from the sets $\{c_1,c_2,\ldots,c_k\}$,  $\{c_{k+1},c_{k+2},\ldots,c_{2k}\}$ respectively. Each block $P^i$ needs $k$ colors and each color can be used in at most $r+1$ consecutive blocks of $T$. We reserve the color set $\{c_1,c_2,\ldots,c_k\}$ for each of the blocks $P^1,\ldots,P^{r+1}$, while the color set $\{c_{k+1},c_{k+2},\ldots,c_{2k}\}$ is reserved for each of the blocks $P^{r+2},\ldots,P^{2r+1}$. Observe that there are exactly $k=\ell$ colors for the block $P^i$ with $1\leq i\leq 2r+1$, and we totally use $2k$ colors for the proper coloring of $T$. Then, the claim follows from Proposition \ref{prop:2k}.\medskip

\textbf{Case 2:} $\alpha\geq 1$ and $\ell=0$. By considering a cyclic orientation of $T$, we create a set of palettes of colors. Throughout, the indices of the blocks are considered in modulo $(2r+1)$. We will reserve  $(2r+1)\alpha$ colors for a $k$-coloring of each block $P^i$. Note that each block $P^i$ needs $k$ colors, and each color can be used in at most $r+1$ consecutive blocks. For each $i\in [2r+1]$, we apply to a cyclic iteration as follows; we reserve colors $c_{(i-1)\alpha+1},c_{(i-1)\alpha+2},\ldots,c_{(i-1)\alpha+\alpha}$ for  each block $P^i,\ldots,P^{i+r}$. Observe that each block $P^i$ receives exactly $\alpha.(r+1)=k$ colors, and we totally use $(2r+1)\alpha$ colors for the proper coloring of $T$. Then, the claim follows with the fact $(2r+1)\alpha \leq \ceil[\Big]{\frac{k(2r+1)}{r+1}}$. \medskip

\textbf{Case 3:} $\alpha\geq 1$ and $\ell\geq 1$. We will combine the first two cases in order to prove the claim. We begin with using $(2r+1)\alpha$ colors as in Case 2, and then use $2\ell$ more colors as in Case 1. That is, we first reserve the color palette $\{c_{(i-1)\alpha+1},c_{(i-1)\alpha+2},\ldots,c_{(i-1)\alpha+\alpha}\}$ for each block $P^i,\ldots,P^{i+r}$ and for each $i\in [2r+1]$, and next we reserve  the color palette $\{d_1,d_2,\ldots,d_\ell\}$ for each of the blocks $P^1,\ldots,P^{r+1}$, while the color palette $\{d_{\ell+1},d_{\ell+2},\ldots,d_{2\ell}\}$ is reserved for each of the blocks $P^{r+2},\ldots,P^{2r+1}$.
We totally use $(2r+1)\alpha +2\ell$ colors. Therefore, we conclude that
\begin{align*}
\chi(T)&\leq (2r+1)\alpha +2\ell=(2r+1)\floor[\Big]{\frac{k}{r+1}}+2\left(k-\floor[\Big]{\frac{k}{r+1}}(r+1)\right)\\ 
&= 2k+\Big( (2r+1)-(2r+2) \Big).\floor[\Big]{\frac{k}{r+1}}\\
&= 2k-\floor[\Big]{\frac{k}{r+1}}= \floor[\Big]{\frac{2k(r+1)-k}{r+1}}\leq \floor[\Big]{\frac{(2r+1)k}{r+1}}. 
\end{align*} 
\end{proof}

\begin{lemma}\label{lem:maxcoloring-neq}
Let $T=R_{2r+1}(P^1,P^2,\ldots,P^{2r+1})$  be a tournament and let $\chi(P^i)=k_i$ for each $1\leq i\leq 2r+1$. Suppose $k_s > k_t\geq k_\ell$ for each $\ell\in [2r+1] \setminus \{s,t\}$. Then, $\chi(T)\leq  \max\bigg{\{} \ceil[\Big]{\frac{(2r+1)(k_s+k_t)-1}{2r+2}}, k_s \bigg{\}}.$    
\end{lemma}
\begin{proof}
We follow the same idea used in the proof of Lemma~\ref{lem:maxcoloring-eq}. We  prove the claim by dividing it into two separate cases. Once again the indices of the blocks are considered in modulo $(2r+1)$ throughout. \medskip

\textbf{Case 1:} $k_s\geq 2k_t$. Then, we have $\chi(T)=\chi(P^s)=k_s$. Indeed, we reserve all colors $\{1,2,\ldots,k_t\}$ to each block in the sequence $P^s,P^{s+1},\ldots,P^{s+r}$. On the other hand, colors in $\{k_t+1,k_t+2,\ldots,2k_t\}$ are reserved for the blocks $P^s,P^{s-1},\ldots,P^{s-r}$. Finally, we reserve  $k_s-2k_t$ more colors for the block $P^s$. Observe that each $P^i$ with $i\in [2r+1]\setminus \{s\}$ receives exactly $k_t$ colors and the block $P^s$ receives exactly $k_s$ colors.
This creates a legal coloring, since each color is used in at most $r+1$ consecutive blocks; hence, $\chi(T)=k_s$. \medskip

\textbf{Case 2:} $k_t<k_s< 2k_t$. We may assume without loss of generality $s=1$ so that $k_s=k_1=\chi(P^1)$. Our aim is to provide a coloring in which $P^1$ receives $k_s$ colors, and each of the other blocks gets $k_t$ colors. Let  $\ell=k_s-\lfloor\frac{k_s}{r+1}\rfloor(r+1)$ and $p_i=\lfloor\frac{k_s}{r+1}\rfloor (i-r-1)$.  Note that $\ell$ is the remainder of $k_s$ divided by $r+1$, and so $\ell \leq r$.

Consider $r+1$ consecutive blocks ending with $P^s=P^1$ in clockwise order, which are $P^{r+2},P^{r+3},$ $\ldots,P^{2r+1},P^1$. We first reserve $k_s$ colors to $P^1$ as follows; start with the color $c_j$ for $1\leq j\leq r+1$ and reserve it for each of the block $P^{r+j+1},P^{r+j+2},\ldots,P^j$. We continue this process until the palette for $P^1$ gets $k_s$ colors, which consumes $\lfloor\frac{k_s}{r+1}\rfloor$-rounds if $r+1$ divides $k_s$. Otherwise, we continue to reserve the color $c_{\lfloor\frac{k_s}{r+1}\rfloor+j}$ with $1\leq j\leq \ell$ for the palette of each of the block $P^{r+j+1},P^{r+j+2},\ldots,P^j$. 

Now, the palette for $P^1$ contains $k_s$ colors within a color set $\{c_1,c_2\ldots,c_{k_s}\}$. We remark that if $\ell=0$, then for each $r+2 \leq i \leq 2r+1$, the palette for the block $P^i$ receives $p_i$ colors, while for each $2 \leq i \leq r+1$, the palette for the block $P^i$ gets $p_{2r+3-i}$ colors. On the other hand, if $ \ell\geq 1$, then  the palette for the block $P^{2r+1}$ receives exactly $\lfloor\frac{k_s}{r+1}\rfloor r+\ell$ colors, and by the symmetry, the block $P^{2}$ is colored with $\lfloor\frac{k_s}{r+1}\rfloor r+\ell-1$ colors. In more detail, for $\ell\geq 0$,  the number of colors in each palette for the blocks are

\begin{itemize}
\item  $k_s$ colors in the palette of $P^{1}$, 
\item  $p_{2r+3-i}+\ell+1-i$ colors in the palette of $P^{i}$ for each $2\leq i \leq \ell+1$, 
\item  $p_{2r+3-i}$ colors in the palette of $P^{i}$ for each $\ell+2\leq i \leq r+1$, 
\item $p_{i}+i-r-1$ colors in the palette of $P^{i}$ for each $r+2\leq i\leq r+\ell+1$, 
\item $p_{i}+\ell$ colors in the palette of $P^{i}$ for each $r+\ell+2\leq i\leq 2r+1$. 
\end{itemize}

Let $\lambda = k_t-\lfloor\frac{k_s}{r+1}\rfloor$. We now reserve colors to the blocks that need extra colors. Observe that the block which needs the largest number of colors is $P^{r+1}$. This is because it has $ \lfloor\frac{k_s}{r+1}\rfloor$ colors for the moment. Notice that if $\ell=0$, then $P^{r+2}$ has  $ \lfloor\frac{k_s}{r+1}\rfloor$ colors as well.  So far we have already used only $k_s$ colors, and to extend this to a proper coloring, we need to create $\lambda=k_t-\lfloor\frac{k_s}{r+1}\rfloor$ more colors in order to obtain a $k_t$ coloring of all blocks except for $P^1$.

Next, consider $r+1$ consecutive blocks starting with $P^{2}$ in clockwise order, which are $P^{2},P^{3},\ldots,P^{r+2}$. We reserve $\lambda$ colors to $P^{r+1}$ as follows; start with the color $d_j$ for $1\leq j\leq r$ and reserve it for each of the block $P^{j+1},P^{j+2},\ldots,P^{j+r+1}$. We continue this process until the palette for $P^{r+1}$ gets $\lambda$ colors, which consumes $\lfloor\frac{\lambda}{r}\rfloor$-rounds if $r$ divides $\lambda$. Otherwise, we continue to reserve the color $d_{\lfloor\frac{\lambda}{r}\rfloor r+j}$ for $1\leq j\leq \lambda-\lfloor\frac{\lambda}{r}\rfloor r$ for the palette of each of the block $P^{j+1},P^{j+2},\ldots,P^{j+r+1}$.

Finally, the palette for $P^{r+1}$ gets $\lambda=k_t-\lfloor\frac{k_s}{r+1}\rfloor$ more colors from the set $\{d_1,d_2\ldots,d_{\lambda}\}$. By the cyclic ordering, the palette for $P^{r+2}$ also gets $\lambda$ more colors.  We conclude that each block except for $P^1$ totally gets $k_t$ colors. In particular, we totally use $k_s+k_t-\lfloor\frac{k_s}{r+1}\rfloor$ colors. Since $k_t\leq k_s-1$, it follows that
\begin{align*}
\chi(T)&\leq  k_s+k_t-\floor[\Big]{\frac{k_s}{r+1}} \leq \floor[\Big]{k_s+k_t-\frac{k_s}{r+1}} \\
 &\leq \floor[\Big]{\frac{(r+1)k_s+(r+1)k_t-k_s}{r+1}}\leq \floor[\Big]{\frac{2rk_s+(2r+2)k_t}{2r+2}}\leq \floor[\Big]{\frac{(2r+1)k_s+(2r+1)k_t-1}{2r+2}}.
\end{align*}
This completes the proof.
\end{proof}

\begin{proposition}\label{prop:funct-max}
For  $r\in \N$, $p \in \R$ with  $a=\frac{r+1}{2r+1}$ and $p\geq 1$,  the following inequality holds
$$a^{\log(1+p)}+a^{\log(1+\frac{1}{p})}\leq 2a $$
where the logarithm has two as its base.
\end{proposition}
\begin{proof}
Let $r\in \N$, $p \in \R$ with  $a=\frac{r+1}{2r+1}$ and $p\geq 1$. Clearly $\frac{1}{2}<a\leq\frac{2}{3} $. For a fixed $a$, we define the function$f(p)=a^{\log(1+p)}+a^{\log(1+\frac{1}{p})} $ on the interval $[0,\infty)$. The function $f$ obtains it maximum value at the point $p=1$. Indeed, the function $f$ is ascending on the interval $(0,1)$ and descending on the interval $(1,\infty)$. Therefore, $p=1$ is the maximum value point of $f$. Hence the claimed equality holds, since $f(p)\leq f(1)=2a$.
\end{proof}

We are now ready to prove our second main result.
\begin{theorem}\label{thm:col-bound}
If $T=R_{2r+1}(P^1,P^2,\ldots,P^{2r+1})$ is a $(2r+1)$-sectionable tournament, then 
$$\chi(T)\leq 2 \left(2- \frac{1}{r+1} \right)^{\log(\dim(T)+1)}-1$$
where the logarithm has two as its base.
\end{theorem}

\begin{proof}
Suppose that $\dim(T)=d$. We use induction on the order of $T$. If $|V(T)|=3$, then $T$ is just a $3$-dicycle so that $\chi(T)=2=2\frac{3}{2}-1$. For the case $|V(T)|>3$, we write $\dim(P^i)=d_i$, $\chi(P^i)=k_i$, $|V(P^i)|=2r_i+1$, and assume that $\chi(P^i)\leq 2 \big( \frac{2r_i+1}{r_i+1} \big)^{\log(d_i+1)}-1$ 
for each $1\leq i\leq m=2r+1$. Observe that the claim trivially holds if $\chi(T)\leq \chi(P^i)$ for some $i\in [2r+1]$. \medskip

Assume that $\max\{\chi(P^i)+\chi(P^j)\}=k_s+k_t$ for $i,j\in [2r+1]$ with $i\neq j$, and let $k_s\geq k_t$. By Corollary \ref{cor:dim}, we have $d\geq \underset{i,j \in [2r+1]}{\max} \{d_i+d_{j}+1 \}$.  We set $r:=\max\{r_1,r_2,\ldots,r_m\}$. It follows $d_s+d_t+1\leq d$ and we set $\min\{d_s,d_t\}=d_\ell$. Note that if $k_s=k_t$, then 
$\chi(P^s)\leq 2 \big( \frac{2r+1}{r+1} \big)^{\log(d_s+1)}-1$, and similarly $\chi(P^t)\leq 2 \big( \frac{2r+1}{r+1} \big)^{\log(d_t+1)}-1$. \medskip 

\textbf{Case $1$.} $k_s=k_t$. We apply to Lemma \ref{lem:maxcoloring-eq}.
\begin{align*}
\chi(T)&\leq \ceil[\Big]{ \frac{(2r+1)k_s}{r+1}} =   \floor[\Big]{ \frac{(2r+1)k_s+r}{r+1}}\leq  \floor[\Big]{ \frac{2r+1}{r+1} \left( 2 \left( \frac{2r+1}{r+1} \right)^{\log(d_s+1)}-1 \right)+\frac{r}{r+1}} \\
&\leq \floor[\Big]{ \frac{2r+1}{r+1} \left( 2 \left( \frac{2r+1}{r+1} \right)^{\log(d_\ell+1)}-1 \right)+\frac{r}{r+1}} \\
&\leq \floor[\Big]{ 2 \left( \frac{2r+1}{r+1} \right)^{1+\log(d_\ell+1)}-\frac{2r+1}{r+1}+\frac{r}{r+1}}  \\
&\leq \floor[\Big]{ 2 \left( \frac{2r+1}{r+1} \right)^{\log(2d_\ell+2)}-1 } \\
&\leq \floor[\Big]{ 2 \left( \frac{2r+1}{r+1} \right)^{\log(d+1)}-1 }. \\
\end{align*}

\textbf{Case $2$.}  $ k_t<k_s$. %We assume without loss of generality that $d_s\geq d_t$.  
We set $w_s:=\log(\frac{d+1}{d_s+1})$ and $w_t:=\log(\frac{d+1}{d_t+1})$, and apply to Lemma~\ref{lem:maxcoloring-neq}. %Notice that $w_s\leq  1\leq w_t $. 
\begin{align*}
\chi(T)&\leq \ceil[\Big]{\frac{(2r+1)(k_s+k_t)-1}{2r+2}} =   \floor[\Big]{\frac{(2r+1)(k_s+k_t)+2r}{2r+2}} \\
&\leq  \floor[\Big]{ \frac{2r+1}{2r+2} \left( 2 \left( \frac{2r+1}{r+1} \right)^{\log(d_s+1)}-1 +2 \left( \frac{2r+1}{r+1} \right)^{\log(d_t+1)}-1\right)+\frac{r}{r+1}} \\
&\leq  \floor[\Big]{ \frac{2r+1}{2r+2} \left( 2 \left( \frac{2r+1}{r+1} \right)^{\log(d_s+1)+w_s-w_s} +2 \left( \frac{2r+1}{r+1} \right)^{\log(d_t+1)+w_t-w_t}\right)-\frac{2r+1}{r+1}+\frac{r}{r+1}} \\
&\leq  \floor[\Big]{ \frac{2r+1}{2r+2} \left( 2\left( \frac{2r+1}{r+1} \right)^{\log(d+1)-w_s} +2 \left( \frac{2r+1}{r+1} \right)^{\log(d+1)-w_t}\right)-1} \\
&\leq  \floor[\Big]{ \frac{2r+1}{2r+2}  \left( \left(\frac{2r+1}{r+1}\right)^{-w_s}+ \left(\frac{2r+1}{r+1}\right)^{-w_t}\right)  2 \left( \frac{2r+1}{r+1} \right)^{\log(d+1)} -1} \\
&\leq  \floor[\Big]{ \frac{2r+1}{2r+2}  \left( \left(\frac{r+1}{2r+1}\right)^{\log\left(\frac{d+1}{d_s+1}\right)}+ \left(\frac{r+1}{2r+1}\right)^{\log\left(\frac{d+1}{d_t+1}\right)}\right)  2 \left( \frac{2r+1}{r+1} \right)^{\log(d+1)} -1}. 
\end{align*}

The claim follows once if we prove that

$$  \left(\frac{r+1}{2r+1}\right)^{\log\left(\frac{d+1}{d_s+1}\right)}+ \left(\frac{r+1}{2r+1}\right)^{\log\left(\frac{d+1}{d_t+1}\right)}\leq \frac{2r+2}{2r+1}. $$

If we write $p=\frac{d_s+1}{d_r+1} $, it follows that 

\begin{align*}
\left(\frac{r+1}{2r+1}\right)^{\log\left(\frac{d+1}{d_s+1}\right)}+ \left(\frac{r+1}{2r+1}\right)^{\log\left(\frac{d+1}{d_t+1}\right)}&\leq \left(\frac{r+1}{2r+1}\right)^{\log\left(\frac{d_s+d_t+2}{d_s+1}\right)}+ \left(\frac{r+1}{2r+1}\right)^{\log\left(\frac{d_s+d_t+2}{d_t+1}\right)}\\
&=\left(\frac{r+1}{2r+1}\right)^{\log\left(1+\frac{d_t+1}{d_s+1}\right)}+ \left(\frac{r+1}{2r+1}\right)^{\log\left(1+\frac{d_s+1}{d_t+1}\right)}\\
&=\left(\frac{r+1}{2r+1}\right)^{\log\left(1+\frac{1}{p}\right)}+ \left(\frac{r+1}{2r+1}\right)^{\log(1+p)} \leq \frac{2r+2}{2r+1},
\end{align*}

where the last equality is from Proposition \ref{prop:funct-max}. Thus, we have

$$  \left(\frac{r+1}{2r+1}\right)^{\log\left(\frac{d+1}{d_s+1}\right)}+ \left(\frac{r+1}{2r+1}\right)^{\log\left(\frac{d+1}{d_t+1}\right)}\leq \frac{2r+2}{2r+1}, $$

which completes the proof.
\end{proof}

\section*{Acknowledgement}
I would like to thank Yusuf Civan for his invaluable comments and generous encouragement during preparation of this manuscript. I have received indispensable help from the referee. S/he pointed out several inconsistencies in the earlier version of this paper and has contributed substantially to the overall improvement of the presentation. 

\section*{Data Availability}

Data sharing not applicable to this article as no datasets were generated or analysed during the current study.

%%%%%%%%%%%%%%%%%%%%%%%%%%%%%%%%%%%%%%%%%%%%%%%%%%%%%%%%%%%%%%%%%%%%%%%%%%%%%%%%%%%%%%%%%%%%%%%%
%%%%%%%%%%%%%%%%%%%%%%%%%%%%%%%%%%%%%%%%%%%%%%%%%%%%%%%%%%%%%%%%%%%%%%%%%%%%%%%%%%%%%%%%%%%%%%%%

\end{document}